\documentclass[reqno]{amsart}
\usepackage[margin=3cm]{geometry}
\usepackage{amsmath}
\usepackage{amssymb}
\usepackage{amsbsy}
\usepackage[latin1]{inputenc}
\usepackage{dsfont}

\usepackage{enumerate}

\def\v        {{\boldsymbol v}}
\def\z        {{\boldsymbol z}}
\def\0        {{\boldsymbol 0}}
\def\n        {{\boldsymbol n}}

\def\u        {{\boldsymbol u}}
\def\w        {{\boldsymbol w}}
\def\e        {{\boldsymbol e}}

\def\f        {{\boldsymbol f}}

\def\H        {{\boldsymbol H}}
\def\V        {{\boldsymbol V}}
\def\L        {{\boldsymbol L}}

\def\R {{\mathds R}}
\newtheorem{theorem}{Theorem}
\newtheorem{lemma}[theorem]{Lemma}

\newtheorem{corollary}[theorem]{Corollary}
\newtheorem{definition}[theorem]{Definition}

\newtheorem{remark}[theorem]{Remark}

\begin{document}

\title[From the Navier-Stokes-$\alpha$ equations to the Navier-Stokes equations]{On the approximation of turbulent fluid flows by the Navier-Stokes-$\alpha$ equations on bounded domains}
\author{J. V. Guti\'errez Santacreu$^\dag$} 
\author{ M. A. Rojas-Medar$^\ddag$}
\thanks{$\dag$ Dpto. de Matemática Aplicada I, E. T. S. I. Informática, Universidad de Sevilla. Avda. Reina Mercedes, s/n. E-41012 Sevilla, Spain. Tel.: + 34 95 455 27 95, Fax: + 34 95 455 78 78,  E-mail: {\tt juanvi@us.es}. JVGS was partially supported by the José Castillejo Spanish grant No. JC2011-0418 from Ministerio de Educación y Ciencia}

\thanks{$\ddag$ Grupo de Matem\'aticas Aplicadas, Dpto. de
Ciencias B\'asicas, Facultad de Ciencias, Universidad del
B\'{\i}o-B\'{\i}o, Campus Fernando May, Casilla 447, Chill\'an,
Chile. E-mail: {\tt marko@ueubiobio.cl}. MARM was partially supported by the Fondecyt-Chile grant No. 1120260 and 121909 GI/C-UBB, Chile}

\date{\today}

\begin{abstract} 
The Navier-Stokes-$\alpha$ equations belong to the family of LES (Large Eddy Simulation) models whose fundamental idea is to capture the influence of the small scales on the large ones without computing all the whole range present in the flow. The constant $\alpha$ is a regime flow parameter that has the dimension of the smallest scale being resolvable by the model. Hence, when $\alpha=0$, one recovers the classical Navier-Stokes equations for a flow of viscous, incompressible, Newtonian fluids. Furthermore, the Navier-Stokes-$\alpha$ equations can also be interpreted as a regularization of the Navier-Stokes equations, where $\alpha$ stands for the regularization parameter.

In this paper we first present the Navier-Stokes-$\alpha$ equations on bounded domains with no-slip boundary conditions by means of the Leray regularization using the Helmholtz operator. Then we study the problem of relating  the behavior of the Galerkin approximations for the Navier-Stokes-$\alpha$ equations to that of the solutions of the Navier-Stokes equations on bounded domains with no-slip boundary conditions. The Galerkin method is undertaken by using the eigenfunctions associated with the Stokes operator. We will derive local- and global-in-time error estimates measured in terms of the regime parameter $\alpha$ and the eigenvalues. In particular, in order to obtain  global-in-time error estimates, we will work with the concept of stability for solutions of the Navier-Stokes equations in terms of the $L^2$ norm. 

\end{abstract}
\maketitle



\section{Introduction}
LES models have rapidly emerged as successful turbulent models for simulating dynamics of fluid flows at high Reynolds numbers ($Re$). These are widely used to solve intensive problems in a great variety of application areas in natural and technical sciences. The starting point is the physical fact that the larger scales of turbulent flows contain most of the kinetic energy of the system, which is transferred to smaller scales via the nonlinear term by an inertial and essentially inviscid mechanism. This process continues creating smaller and smaller scales until forming eddies in which the viscous dissipation of energy finally takes place. Therefore, the small-scale dynamics can sometimes have an influence on large-scale structures and hence affect the overall behavior of a fluid flow in many physical phenomena. But computing all of the degrees of freedom required to describe a flow in its entirety at a high Reynolds number  turns out to be impossible to achieve due to considerable limitations in computing power.  It is conjectured by  Kolmogorov's scaling theory that the number of degrees of freedom required by a direct numerical simulation of the Navier-Stokes equations is of the order of $Re^{\frac{9}{4}}$. This theory assumes that the turbulent fluid flow is universal, isotropic and statistically homogeneous for the small-scale structures at high Reynolds numbers. LES approaches avoid such a  situation  by computing large-scale turbulent structures in the fluid flow while the effect of the small-scale ones are modeled. In the literature there exist several ways of separating large scales from small ones. Some examples are regularization techniques such as the Navier-Stokes-$\alpha$ equations and closely related models \cite{Chen-et-al_I, Chen-et-al_II, Chen-et-al_III, Cheskidov-Holm-Olson-Titi, Ilyin-Lunasin-Titi}, nonlinear viscosity methods such as the Smagorinsky model \cite{Smagorinsky}, spectral eddy-viscosity methods such as the Kraichnan model \cite{Kraichnan}, and sub-grid methods such as variational multi-scale models \cite{Hughes-Mazzei-Oberai-Wray, Hughes-Oberai-Mazzei, Hughes-Mazzei-Jensen}.

The emphasis of this work is focused on the Navier-Stokes-$\alpha$ equations. They can be derived in three different ways. 
\begin{enumerate} [(i)]
\item 
Firstly, these equations appeared as a generalization of the Euler-$\alpha$ equations by adding an {\it ad hoc} viscous term \cite{Chen-et-al_I, Chen-et-al_II, Chen-et-al_III} whose explicit form was motivated by physical arguments in absence of boundaries. The Euler-$\alpha$ equations were derived from Lagrangian averaging and asymptotic expansions in Hamilton's principle to the turbulence in the flow being statistically homogeneous and isotropic \cite{Holm-Marsden-Ratiu-1998-I, Holm-Marsden-Ratiu-1998-II}. The viscous term can be also derived from a stochastic interpretation of the Lagrangian flow maps for domains with boundary \cite{Marsden-Shkoller, Shkoller}.  
\item
Secondly, the Navier-Stokes-$\alpha$ equations can be seen as a Leray regularization of the Navier-Stokes equations by using the Helmholtz operator \cite{Guermond-Oden-Prudhomme}. In order to get the resulting system of PDEs to be Galilean invariant, the convective term must be written in its rotational form. On the other hand, the property of being Galilean invariant does not hold for other $\alpha$-models such as the Leray-$\alpha$ equations \cite{Guermond-Oden-Prudhomme}. 

In general, the Leray regularization approach supplies systems of PDEs which are well-posed, as occurs with the Navier-Stokes-$\alpha$ equations. That is, the fundamental mathematical questions of existence, uniqueness and stability for the Navier-Stokes-$\alpha$ equations are known; in particular uniqueness is even proved for three-dimensional domains. Unfortunately, the uniqueness question of global-in-time solutions of the three-dimensional Navier-Stokes equations has not been  solved yet. This issue is intimately related to the one of whether or not the Navier-Stokes equations are a suitable model for turbulent fluids.     

\item 
Finally, the Rivlin-Ericksen continuum theory of differential type gives similar models to the Navier-Stokes-$\alpha$ equations for describing dynamics of a number of non-Newtonian fluids (such as water solution of polymers). These fluids are characterized because its stress-deformation response does not depend only on the constitutively indeterminate pressure and the stretching tensor but also certain other kinematic tensors called the Rivlin-Ericksen stress tensors. Among fluids of different type, one finds the grade-$n$ fluids  whose stress tensor is a polynomial of degree $n$ in the first $n$ Rivlin-Ericksen stress tensor. We refer to \cite{Rivlin-Ericksen, Truesdell-Noll, Dunn-Fosdick, Dunn-Rajagopal} and the references therein for the derivation of the grade-$n$ fluid equations and further physical background on the continuum theory of differential type. Surprisingly, the grade-two fluid equations resembles the Navier-Stokes-$\alpha$ equations except for the viscous dissipation being weaker in the former. It seems to be that the grade-two fluid equations does not in fact provides the correct dissipation for approximating turbulent phenomena near the wall but instead they present the same hyperstress as in the Navier-Stokes-$  \alpha$ equations \cite{Marsden-Shkoller, Shkoller}. That is, the inviscid case of the grade-two fluid equations coincides with the Euler-$\alpha$ equations.
\end{enumerate}

A key property determining the long time  behavior of many evolutionary partial differential equations is the dissipation of energy. In particular, dissipativity is central to the existence of a global attractor. The concept of the global attractor is closely related to that of turbulence. In a nutshell, the global attractor is a compact set in the phase space that absorbs all the trajectories starting from any bounded set after a certain time. Therefore, the global attractor retains the long-time behavior of the whole dynamics of the fluid flow. Unsurprisingly, the dimension of the global attractor is related to the number of degrees of freedom needed to capture the smallest dissipative structures of the flow according to Kolmogorov's theory. 

In this work we are interested in the properties of the Navier-Stokes-$\alpha$ equations in the limit  as $\alpha$ approaches zero. In particular, we will study the properties of the Galerkin solutions of the Navier-Stokes-$\alpha$ equations and their relations with the solutions of the Navier-Stokes equations. The Galerkin approximation is performed by using the eigenfunctions associated to the Stokes operator. We will show local- and global-in-time error estimates\footnote{By abuse of nomenclature, we use  local- and global-in-time  estimates to make reference to estimates on $[0,T]$ for $0<T<\infty$ and $T=\infty$, respectively.} in the $L^\infty(0,T; \L^2(\Omega))$ norm, for $0<T<\infty$ and $T=\infty$,  between the Galerkin approximation of the Navier-Stokes-$\alpha$ equations and the solution of the Navier-Stokes equations in terms of the eigeinvalues and the parameter $\alpha$.     
It is widely believed that global-in-time error estimates should not hold without assuming any additional property of the solution of the Navier-Stokes equations. Even if one assumes global-in-time bounds for the solution being approximated, the best general error estimates predict an asymptotically increasingly accurate approximation as time goes to $\infty$. In order to avoid such an undesirable circumstance one must introduce the concept of \textit{stability} for solutions of the Navier-Stokes equations related to the decay of perturbations at infinite. This way we will be able to prove that the Galerkin solution approximates the exact solution uniformly in time, even if such a solution reaches the global attractor, without losing accuracy. 

The remainder of this paper is organized as follows. We present the Navier-Stokes-$\alpha$ equations on bounded domains with no-slip boundary conditions by means of the Leray regularization using the Helmholtz operator in Section $2$.  In Section $3$, we introduce some short-hand notation and cite some useful known results. In Section $4$, we give a brief overview of the mathematical results presented in this paper. Section $5$ studies local-in-time error estimates. This is broken into two subsections. In Section $5.1$, local-in-time a priori energy estimates are established for the Galerkin approximations and for the solution to be approximated of the Navier-Stokes equations as a consequence of passing to the limit. Then Theorem \ref{Th1} is proved in Section $5.2$. Section $6$ is devoted to demonstrating global-in-time error estimates. We again broke this section into four subsections. In Section $6.1$, global-in-time a priori energy estimates for the Galerkin approximations are showed. In Section $6.2$ the notion of perturbations in the $\L^2(\Omega)$ sense is introduced. Auxiliary results are presented in Section $6.3$. Then Theorem \ref{Th2} is demonstrated in Section $6.4$. In Section $7$  we end up with several concluding remarks.


\section{The model}
The Navier-Stokes equations for the flow of a viscous, incompressible, Newtonian fluid can be written as
\begin{equation}\label{NS-PDE}
\left\{
\begin{array}{rclcc}
\displaystyle
\partial_t\u-\nu \Delta \u + (\u\cdot \nabla) \u  +
\nabla p &=&  \f & \mbox{ in }& \Omega\times (0, T), 
\\
\nabla\cdot \u&=&0 &\mbox{ in } & \Omega\times (0, T),
\end{array}
\right.
\end{equation}
with $ \Omega$ being a bounded domain of $\R^d$, $d=2$ or $3$, and with $0<T<+\infty$ or $T=+\infty$. Here $\u: \Omega\times (0,T)\to \R^d$ represents the incompressible fluid velocity and $p: \Omega\times (0,T)\to \R $ represents the fluid pressure. Moreover, $\f$ is the external force density which acts on the system, and $\nu > 0$ is the kinematic fluid viscosity.

These equations are supplemented by the no-slip boundary condition
\begin{equation}\label{boundary-condition}
\u={\boldsymbol 0}\quad \mbox { on } \quad  \partial \Omega\times (0,T),
\end{equation}
and the initial condition
\begin{equation}\label{initial-condition}
\u(0)=\u_0\quad\mbox{ in }\quad \Omega. 
\end{equation}

Next we will present the Navier-Stokes-$\alpha$ equations on bounded domains by using the Leray approach with the Helmholtz regularization \cite{Guermond-Oden-Prudhomme}. First of all, we write
$$
(\u\cdot\nabla)\u=-\u \times (\nabla\times\u)+\frac{1}{2}\nabla(\u\cdot\u).
$$
Then system (\ref{NS-PDE}) reads as 
$$
\left\{
\begin{array}{rclcc}
\displaystyle
\partial_t\u-\nu \Delta \u - \u\times(\nabla\times\u) +
\nabla p' &=&  \f & \mbox{ in }& \Omega\times (0, T), 
\\
\nabla\cdot \u&=&0 &\mbox{ in } & \Omega\times (0, T),
\end{array}
\right.
$$
where $p'=p+\frac{1}{2}\nabla(\u\cdot\u)$. Next we apply the Leray regularization with the Helmholtz operator to find
\begin{equation}\label{NS-alpha}
\left\{
\begin{array}{rclcc}
\displaystyle
\partial_t\u-\nu \Delta \u -\v\times(\nabla\times\u) +
\nabla p' &=&  \f & \mbox{ in }& \Omega\times (0, T), 
\\
\nabla\cdot \u&=&0 &\mbox{ in } & \Omega\times (0, T),
\end{array}
\right.
\end{equation}
where $\v$ is defined as
\begin{equation}\label{Helmholtz}
\left\{
\begin{array}{rcccl}
\v-\alpha^2\Delta\v+\nabla\pi&=&\u& \mbox{ in }&\Omega\times (0,T), 
\\ 
\nabla\cdot \v&=&0& \mbox{ in }&\Omega\times (0,T), 
\\
\v&=&0&\mbox{ on }&\partial\Omega\times (0,T), 
\end{array}
\right.
\end{equation}
with $\alpha>0$ being the regularization parameter. 

In the definition of the pair $(\v, \pi) $ we observe the first difference between the periodic and non-periodic case. 
For periodic domains, this null-space of the Laplacian is only made of constant functions; therefore, working in mean-free spaces, one finds that $\pi\equiv0$. Hence, the Stokes and the Laplace operator do coincide, apart from the domain of definition. Instead, for non-periodic domains, the pseudo-pressure $\pi$ is used to rule out a much wider class of functions; so the Stokes and the Laplace operator are different.   

System (\ref{NS-alpha})-(\ref{Helmholtz}) together with (\ref{boundary-condition}) and (\ref{initial-condition}) is called the Cauchy problem for the Navier-Stokes-$\alpha$ equations on boundary domains with no-slip boundary conditions. It is clear that if one considers $\alpha=0$, one recovers the Cauchy problem for the Navier-Stokes equations. 

One may rewrite (\ref{NS-alpha}) in terms of $\v$ only, by a  direct substitution, and so one finds the original Navier-Stokes-$\alpha$ system of PDEs:    
\begin{equation}\label{NS-alpha-II}
\left\{
\begin{array}{rclcc}
\displaystyle
\partial_t(\v -\alpha^2\Delta \v)-\nu \Delta(\v-\alpha^2\Delta\v)- \v\times(\nabla\times(\v-\alpha^2 \Delta \v))  +
\nabla p'' &=&  \f & \mbox{ in }& \Omega\times (0, T), 
\\
\nabla\cdot \v&=&0 &\mbox{ in } & \Omega\times (0, T),
\end{array}
\right.
\end{equation}
where $p''=p'+\partial_t\pi+\Delta\pi$. Observe that $\nabla\times\nabla \pi=0$ have been used.


In \cite{Chen-et-al_I, Chen-et-al_II, Chen-et-al_III, Foias-et-al} system (\ref{NS-alpha-II}) was derived on domains that do not have a boundary. For that reason, system (\ref{NS-alpha-II}) is typically studied in the absence of boundary conditions (e.g. in the $d$-dimensional torus $\Omega= \mathds{T}^d$ or  the whole space $\Omega=\R^d$). This sort of domains are less physical interest but provide sometimes a convenient slightly simplified model which decouples the equations from the boundary and makes easier somewhat the mathematical analysis. But, boundaries are of importance in many engineering applications. 

A key reason that system (\ref{NS-alpha})-(\ref{Helmholtz}) is preferred over system (\ref{NS-alpha-II}) on bounded domains is the fact that system (\ref{NS-alpha-II}) needs to be complete with an extra boundary condition for $-\Delta \u$ due to the presence of the bi-Laplacian operator. 
At this point we need to make two observations regarding such a boundary condition because some care must be taken in choosing it. Introducing a boundary condition for $-\Delta\u$ may lead to either the initial boundary-value problem for (\ref{NS-alpha-II}) being ill-posed or   phenomena near the wall being unrealistic. For instance, one may consider homogeneous Dirichlet boundary conditions for both $\v$ and $\Delta\v$, i.e.,
\begin{equation}\label{boundary-condition-bi-Laplacian}
\v={\boldsymbol 0}\quad \mbox{ and }\quad -\Delta\v={\boldsymbol 0}\quad \mbox { on } \quad  \partial \Omega\times(0,T).
\end{equation}
These boundary conditions give rise to an overdetermined problem \cite{Ladyzhenkaya} due to the incompressibility condition. It is well to highlight, here, that the boundary conditions to be imposed for (\ref{NS-alpha}) and (\ref{Helmholtz}) are $\u=\v={\boldsymbol 0}$ on $\partial\Omega\times (0,T)$ or equivalently $\u=A\u={\boldsymbol 0}$ on $\partial\Omega\times (0,T)$, where $A$ stands for the Stokes operator. The reader is referred to \cite{Shkoller, Marsden-Shkoller} for a detailed discussion of the boundary conditions for the Navier-Stokes-$\alpha$ equations on bounded domains.  It is important to observe that system (\ref{NS-alpha})-(\ref{Helmholtz}) is totally equivalent to the one presented in \cite{Shkoller, Marsden-Shkoller}. 


As discussed in Section 1, there is a connection between the Navier-Stokes-$\alpha$ equations and  the grade-two fluid equations, which are  (\ref{NS-alpha-II}) with $-\nu\Delta\v$ rather than $-\nu\Delta(\v-\alpha^2\Delta\v)$, derived from the continuum mechanical principle of material frame-indifference \cite{Truesdell-Noll}. In this context, the constant $\alpha$ is a material parameter measuring the elastic response of the fluid. The sign of $\alpha$ is determined by applying the Clausius-Duhem inequality together with the fact that the free energy must have a stationary point in equilibrium \cite{Dunn-Fosdick} so that the grade-two fluids are compatible with thermodynamics. We refer the reader to \cite{Dunn-Rajagopal} for a detailed discussion on the sign of $\alpha$.
In this case, there is no need of any extra boundary condition for $-\Delta\u$. 
\subsection{Previous works}
Rautmann \cite{Rautmann} initialized the study of error estimates for the spectral Galerkin approximations of the Navier-Stokes equations. His results were local in time since the bounds have no meaning as time goes to infinity. Heywood \cite{Heywood} noted that further assumptions were necessary in order to yield global-in-time error estimates. This additional assumption concerns {\it stability} of the solution of the Navier-Stokes equations. Heywood formulated the stability condition in terms of the $\H^1(\Omega)$ norm and gave global-in-time error estimates in the same norm. Later Salvi \cite{Salvi} obtained global-in-time error estimates in the $\L^2(\Omega)$-norm by assuming stability in the same norm.    

Similar programs to that of this work were performed for the density-dependent Navier-Stokes equations \cite{Braz-Rojas} and the Kazhikhov-Smagulov equations \cite{Gutierrez-Rojas}. Global-in-time error estimates for the Galerkin approximations were derived in $\H^1(\Omega)$ for the velocity under the assumption of stability in the $\H^1(\Omega)$ norm. The density, in both models, plays an important role in defining the concept of stability.

Foias et al. proved the global-in-time existence and uniqueness of regular solutions to the Navier-Stokes-$\alpha$ equations with periodic boundary conditions in \cite{Foias-Holm-Titi}. Later in \cite{Marsden-Shkoller} Marsden and Shkoller established the same results on domains with boundary. 
  
The first convergence analysis between the Navier-Stokes-$\alpha$ and the Navier-Stokes equations as $\alpha$ approaches to zero was undertaken in \cite{Foias-Holm-Titi}. There it was established that there exists a subsequence for which the regular solutions of the Navier-Stokes-$\alpha$ equations converge strongly in the $L^2_{\rm loc}(0,\infty; \L^2(\mathds{T}^3))$ norm to a weak solution of the Navier-Stokes equations. In this work no convergence rate was provided. In this sense, in \cite{Chen-Guenther-Kim-Thomann-Waymire}, the convergence rate in the $L^1(0,T; \L^2(\mathds{T}^3))$ norm was proved to be of order $\mathcal{O}(\alpha)$ for small initial data in Besov-type function spaces in which global existence and uniqueness of solutions for the Navier-Stokes equations can be established. But this convergence rate deteriorates as $T$ goes to $\infty$. In \cite{Cao-Titi} the convergence rate of solutions of various $\alpha$-regularization models to weak solutions of the Navier-Stokes equations is given in the $L^\infty(0,T; \L^2(\mathds{T}^2))$ norm being of order of $\mathcal{O}(\alpha (\log \frac{1}{\alpha})^{\frac{1}{2}})$. In addition to these results, error estimates for the Galerkin approximation  of the Leray-$\alpha$ equations were presented in the $L^\infty(0,T; \L^2(\mathds{T}^2))$ norm being of order of $\mathcal{O}(\frac{1}{\lambda_{n+1}} (\log \lambda_{n+1})^{\frac{1}{2}})$, under the assumption $\alpha^2\lambda_{n+1}<1$,  where $\lambda_{n+1}$ is the $(n+1)${th} eigenvalue of the Stokes operator. In particular, the relation between the eigenvalue $\lambda_{n+1}$ and the regularization parameter $\alpha$ means that the dimension of smaller scales, which is captured by the Navier-Stokes-$\alpha$ equations, and the number of degrees of freedom needed to compute the Galerkin approximations are related. The situation would be more favorable if we could avoid such a relation since one can independently approximate either a solution of the Navier-Stokes equations or a solution of  the Navier-Stokes-$\alpha$ equations. This fact is connected with the regularity of the solution being approximated as we will see in this work. As a result of improving the regularity, the logarithmic factor is removed.

The existence of the global attractor for the Navier-Stokes-$\alpha$ equations, as well as estimates for the Hausdorff and fractal dimensions, in terms of the physical parameters of the equations, were established in \cite{Foias-Holm-Titi}.  Vishik et al.  \cite{Vishik-Titi-Chepyzhov} proved the convergence of the trajectory attractor of the Navier-Stokes-$\alpha$ equations to the trajectory attractor of the three-dimensional Navier-Stokes equations as $\alpha$ approaches zero.
\subsection{The contribution of this paper}
Let us highlight the main contribution of this paper and how it differs form existing work. Principally we compare our work with that of Cao and Titi \cite{Cao-Titi}.

\begin{enumerate}
\item The framework in the present paper is that of the Navier-Stokes equations on two-dimensional bounded domains with non-slip boundary conditions. Here one finds the first difference with the work of Cao and Titi \cite{Cao-Titi} which is carried out on the two-dimensional torus with  periodic boundary conditions.

\item First, we directly derive a local-in-time estimate for the error $\u^\alpha_n-\u$ in the $L^\infty(0,T; \L^2(\Omega))$ norm with $\u^\alpha_n$ and $\u$ being the Galerkin approximation of the Navier-Stokes equations and the solution to the Navier-Stokes equations, respectively. Instead, in \cite{Cao-Titi}, this error estimate is obtained in two steps. First, the convergence rate for $\u-\u^\alpha$ is obtained where $\u^\alpha$ is the solution to the Navier-Stokes-$\alpha$ equations. Then, the error estimate for $\u^\alpha_n-\u$ is proved.


\item Our local-in-time error estimate takes the form 
$$
\|\u^\alpha_n(t)-\u(t)\|^2\le K(t)(\lambda_{1}^{-\frac{1}{2}}\alpha^2+\lambda_{n+1}^{-\frac{3}{2}}),
$$
with $K$ being a function with exponential growth in time and depending only on problem data. This error estimate is only optimal with respect to the regularization parameter $\alpha$. See Section~\ref{Remarks} for optimal error estimates with respect to the eigenvalues.  
In \cite{Cao-Titi}, under the assumption $\alpha^2\lambda_{n+1}<1$, the local-in-time error estimate is of the form   
$$
\begin{array}{rcl}
\|\u^\alpha_n(t)-\u(t)\|^2&\le& 2( \|\u(t)-\u^\alpha(t)\|^2+\|\u^\alpha(t)-\u^\alpha_n(t)\|^2)
\\
&\le&\displaystyle K_1(t) \alpha^2 \log \frac{1}{\alpha}+K_2(t)\frac{1}{\lambda_{n+1}^{2}} \log \lambda_{n+1},
\end{array} 
$$
with $K_1$ being a function with exponential growth in time and depending only on problem data.   This error estimate result turns out to be suboptimal with respect to $\alpha$ and $\lambda_{n+1}$. The relation between $\alpha$ and $\lambda_{n+1}$ avoid approximating independently either a solution of the Navier-Stokes-$\alpha$ or the Navier-Stokes equations.

\item It is clear that local-in-time error estimates are meaningless for large time. For that reason, our second result is a global-in-time error estimate which we prove with the help of the stability of solutions of the Navier-Stokes equations. As far as we are concerned, this sort of results is the first time that are addressed in the literature for the Navier-Stokes-$\alpha$ equations.   
\item In proving the local- and global-in-time error estimates we do not use the extra regularity of the Galerkin approximations coming from the hyperviscosity term in the Navier-Stokes-$\alpha$ equations to control the hyperstress term as done in \cite{Cao-Titi}. This gives a hint about how to prove a similar result for the grade-two fluid equations which do not present such a hyperviscosity term.

\end{enumerate}


\section {Notation and preliminaries}
In this section we shall collect some standard notation and preparatory results that will be used throughout this work.


\begin{enumerate}
\item[(H1)] Let $\Omega$ be a bounded domain of $\mathds{R}^2$ whose boundary $\partial\Omega$ is of class $C^{2,1}$, i.e., the boundary $\partial\Omega$ has a finite covering such that in each set of the covering the boundary $\partial\Omega$ is described by an equation $x_N=F(x_1, ..., x_{N-1})$ in some orthonormal basis, with $F$ being a Hölder-continuous function of order $2$ with exponent $1$, and the domain $\Omega$ is on one side of the boundary, say $x_N > F (x_1,...,x_{N-1})$.
\end{enumerate}

We denote by $L^p(\Omega)$, with $1\le p \le \infty$, and $H^m(\Omega)$, with $m\in \mathds{N}$, the usual Lebesgue and Sobolev spaces on $\Omega$ provided with the usual norm $\|\cdot\|_{L^p(\Omega)}$ and $\|\cdot\|_{H^{m}(\Omega)}$ with respect to Lebesgue measure. In the $L^2(\Omega)$ space,  the inner product  and norm are denoted by $(\cdot, \cdot)$ and  $\|\cdot\|$, respectively. Let $ C^\infty_0 (\Omega)$ be functions defined on $\Omega$ and having continuous derivatives of any order with compact support in $\Omega$. Boldfaced letters will be used to denote vector spaces and their elements. We will use $C$, with or without subscripts, to denote generic constants independent of all problem data. Moreover, $E$ and $K$ stand for constants depending on all problem data.   

We now give several function spaces developed in the theory of Navier-Stokes. Thus we denote as  
$$\boldsymbol{\vartheta}=\{\v\in \boldsymbol{C}^\infty_0(\Omega): \nabla\cdot\v=0 \mbox{ in } \Omega \}. $$
Then the spaces $\H$ and $\V$ are the closure in the $\L^2(\Omega)$ and $\H^1(\Omega)$ norm,  respectively, characterized by
$$
\begin{array}{lll}
\H&=& \{ \u \in \L^2(\Omega) : \nabla\cdot\u =0 \mbox{ in } \Omega, \u\cdot\boldsymbol{n} = 0 \hbox{
on }
\partial\Omega \},
\\
{\V}&=& \{\u \in \H^1(\Omega) : \nabla\cdot\u =0 \mbox{ in } \Omega, \u = {\bf 0}
\hbox{ on } \partial\Omega \},
\end{array}
$$
where $\n$ is the outward unit normal vector to $\partial \Omega$. This characterization is valid under $(\rm H1)$. 

Let $-\infty\le a <b\le +\infty$ and let $X$ be a Banach space. Then
$L^p(a,b;X)$ denotes the space of the equivalence class of Bochner-measurable,  $X$-valued functions on $(a, b)$
such that $\int_a^b\|f(s)\|^p_{X} {\rm d} s<\infty$ for $1\le
p<\infty$ or ${\rm ess}\sup_{s\in(a,b)}\|f(s)\|_X<\infty$ for
$p=\infty$. Moreover, 
$H^1(a,b; X)$ is the space of the equivalence class of
$X$-valued functions such that $(\int_a^b\|f(s)\|^2_X+\|\frac{d}{ds}f(s)\|^2_X{\rm d} s)^{1/2}<\infty$.  

We let $P: \L^2(\Omega)\to \H$ be the Helmholtz-Leray orthogonal projection operator and let $A: D(A)\subset \H \to \H$ be the Stokes operator defined as $A=-P\Delta$ where $D(A)=\V\cap\H^2(\Omega)$.

The next lemma is about the stability of the Helmholtz-Leray operator. See \cite[p.18]{Temam}. 
\begin{lemma}\label{le:Temam} For $\u\in \H^1(\Omega)$, $\|P\u\|_{H^1(\Omega)}\le \|\u\|_{H^1(\Omega)}$.
\end{lemma}

The following two lemmas collect some properties of the Stokes operator $A$. For a proof, see e.g. \cite[Chapter 4]{Constantin-Foias}.
\begin{lemma} It follows that:
\begin{enumerate}[$(i)$]
\item The operator $A$ is bijective, self-adjoint, and positive definite.
\item The  operator $A^{-1}$ is injective, self-adjoint, and compact in $\H$.
\item There exist a set of eigenvalues $\{\lambda_n\}_{n=1}^{\infty}$ and a basis of eigenfunctions $\{\w_n\}_{n=1}^{\infty}$ satisfying
\begin{enumerate}
\item $A\w_n=\lambda_n\w_n$ with $\w_n\in D(A)\cap \H^2(\Omega)$.
\item $0<\lambda_1<\cdots\le\lambda_n\le\lambda_{n+1}\le \cdots$.
\item $\lim_{n\to \infty}\lambda_n=\infty$.
\item There exists a constant $C>0$ such that $\lambda_n\ge C n \lambda_1$.

\end{enumerate}
\end{enumerate}
\end{lemma}
Let $\beta>0$. Define the operator $A^\beta:D(A^\beta)\subset \H\to\H$ such that 
$$A^\beta\u=\sum_{n=1}^\infty \lambda_n^{\beta} (\u, \w_n) \w_n,$$
where
$$D(A^\beta)=\{\u\in \H;  \quad \sum_{n=1}^\infty \lambda_n^{2\beta}|(\u,\w_n)|^2<\infty\}.$$ 
Moreover,  the space $D(A^\beta)$ is endowed with the inner product
$$ (A^\beta\u,A^\beta\v)=\sum_{n=1}^\infty\lambda^{2\beta} u_n v_n,$$  
where $u_n=(\u, \w_n)$ and $v_n=(\v, \w_n)$, and the associated norm
$$\|A^{\beta}\u\|^2=\sum_{n=1}^\infty\lambda_n^{2\beta}|(\u, \w_n)|^2.$$ 
In particular,  $D(A^{1/2})=\V$ and $D(A)=\H^2(\Omega)\cap\V$ hold.

\begin{lemma} The set $\{\w_n\}_{n=1}^{\infty}$ is an orthogonal basis of the spaces $\H$, $D(A^{\frac{1}{2}})$, $D(A)$, and $D(A^\frac{3}{2})$  endowed with the inner products $(\cdot,\cdot)$, $(A^\frac{1}{2}\cdot, A^\frac{1}{2}\cdot)$, $(A\cdot, A\cdot)$, and $(A^\frac{3}{2}\cdot, A^\frac{3}{2}\cdot)$, respectively. 

\end{lemma}

It is well-known that the Stokes operator is a maximal monotone operator. Its resolvent $(I+\alpha^2 A)^{-1}$ is well-defined for all $\alpha>0$ and satisfies some properties useful in further developments. We state such properties as a lemma below. See \cite[Chap. 5]{Brezis} for a proof.
\begin{lemma} It follows that: 
\begin{enumerate}[(i)]
\item The operator $(I+\alpha^2 A)^{-1}: \H\to D(A)$ is bounded, linear and self-adjoint with \begin{equation}\label{ResolventL2}
\|(I+\alpha^2 A)^{-1}\|_{{\mathcal L}(\H)}\le 1.
\end{equation}
\item The operator   $ A^{\frac{1}{2}}(I+\alpha^2 A)^{-1}:  \H\to D(A^{\frac{1}{2}})$ is linear and bounded with  
\begin{equation}\label{ResolventH1}
\|A^{\frac{1}{2}}(I+\alpha^2 A)^{-1}\|_{{\mathcal L}(\H ,D(A^{\frac{1}{2}}))}\le 1
\end{equation}
and  
\begin{equation}\label{Resolvent-alphaH1}
\|(\alpha^2A)^{\frac{1}{2}}(I+\alpha^2 A)^{-1}\|_{{\mathcal L}(\H)}\le 1.
\end{equation}
\item The operator  $(\alpha^2 A)(I+\alpha^2 A)^{-1}:  \H\to \H$ is linear and bounded with 
\begin{equation}\label{ResolventH2}
\|(\alpha^2 A)(I+\alpha^2 A)^{-1}\|_{{\mathcal L}(\H)}\le 1.
\end{equation}

\item Furthermore, there holds 
\begin{equation}\label{I-Ja=aAJa}
I-(I+\alpha^2 A)^{-1}=\alpha^2 A (I+\alpha^2 A)^{-1}=\alpha^2(I+\alpha^2 A)^{-1}A.
\end{equation}
\end{enumerate} 
\end{lemma}

The next lemma provides  equivalence of norms between $\|A^{\beta}\cdot\|$ and $\|\cdot\|_{H^m(\Omega)}$.
\begin{lemma}[Poincaré]\label{le:Poincare} If $\u\in D(A^{\frac{3}{2}})$, then 
\begin{equation}\label{Poincare}
\|\u\|\le \lambda_1^{-\frac{1}{2}}\|A^{\frac{1}{2}}\u\|\le\lambda_1^{-1} \|A\u\|
\le \lambda_1^{-\frac{3}{2}} \|A^{\frac{3}{2}}\u\|.
\end{equation}
where $\lambda_1$ is the first eigenvalue of the Stokes operator.

Moreover, there exist two constants $C_1, C_2 >0$ such that
$$
\begin{array}{rcccccl}
C_1 \|A^{\frac{1}{2}}\u\|&\le & \|\u\|_{H^1(\Omega)}& \le & C_2 \|A^{\frac{1}{2}}\u\| &\mbox{ for all } &\u\in D(A^{\frac{1}{2}}), 
\\
C_1 \|A \u\|&\le& \|\u\|_{H^2(\Omega)} &\le& C_2 \|A\u\| &\mbox{ for all }& \u\in D(A),
\\
C_1 \|A^{\frac{3}{2}}\u\|&\le& \|\u\|_{H^3(\Omega)} &\le& C_2 \|A^{\frac{3}{2}}\u\| &\mbox{ for all }&  \u\in D(A^{\frac{3}{2}}).
\end{array}
$$
\end{lemma}

Let us define $\V_n=span\{\w_1,...,\w_n\}$ as the finite vector space spanned by the first $n\in\mathds{N}$ eigenfunctions associated to the Stokes operator. Thus we consider $P_n: \H\to \V_n$ to be the orthogonal projection operator  with respect to the $\L^2(\Omega)$ inner product and $P_n^{\perp}:=I-P_n$ to be the projection onto $\V_n^{\perp}$, the $\L^2(\Omega)$ orthogonal space to $\V_n$. 
 
The following lemma shows elementary properties for $P_n$ and $P^{\perp}_n$ that will be used frequently. We refer the reader to \cite{Rautmann} for a proof.
\begin{lemma}\label{le:errors} Given $\u\in \H$, it follows that
\begin{equation}\label{stabL2}
\|P_n\u\|\le  \|\u\|.
\end{equation}
Moreover, if $\u\in D(A^{\frac{1}{2}})$, then
\begin{equation}\label{errorL2H1}
\|P_n^{\perp}\u\|^2\le \frac{1}{\lambda_{n+1}}\|A^{\frac{1}{2}}P^\perp_n\u\|^2,
\end{equation}
\begin{equation}\label{stabH1}
\|A^{\frac{1}{2}} P_n\u\|\le  \|A^{\frac{1}{2}}\u\|.
\end{equation}
In addition, if $\u\in D(A)$, then
\begin{equation}\label{errorH1H2yL2H2}
\|A^{\frac{1}{2}} P_n^{\perp}\u\|^2\le \frac{1}{\lambda_{n+1}}\|A P^\perp_n\u\|^2\quad\mbox{ and }\quad \|P_n^{\perp}\u\|^2\le \frac{1}{\lambda_{n+1}^2}\|AP^\perp_n\u\|^2,
\end{equation}
\begin{equation}\label{stabH2}
\|A P_n\u\|\le  \|A\u\|.
\end{equation}
\end{lemma}

Let $\u, \v\in \boldsymbol{\vartheta}$. Then we define $B(\u, \v)$ as
$$ B(\u,\v)=P((\u\cdot\nabla)\v)$$
and 
$\widetilde B(\u, \v)$ as 
$$ \widetilde B(\u,\v)=-P(\u\times(\nabla\times\v)).$$ 
Using the fact that 
$$(\u\cdot	\nabla)\v+(\nabla\u)^T\v=-\u\times(\nabla\times\v)+\nabla(\u\cdot\v)$$
and applying the Helmholtz-Leray operator, we get the relation
\begin{equation}\label{rel:tB-B}
B(\u, \v)+B^\star (\u, \v)=\widetilde B(\u, \v),
\end{equation}
where we have denoted $B^\star (\u, \v)=P((\nabla \u)^T\v)$. Moreover,
we have the relation
\begin{equation}\label{rel:B_starB}
( B^\star(\u, \v), \w)=(B(\w, \v), \u)
\end{equation}
for all $\u, \v, \w\in \V$.

Next we review some needed inequalities and continuity properties of the operators $B$ and $\widetilde B$.  
\begin{lemma}\label{le:opB}
The bilinear operator $B$ is continued as follows. There exists a constant $C>0$ scale invariant such that
\begin{enumerate}[(i)]
\item For all $ \u\in D(A^{\frac{1}{2}})$, $\v\in D(A^{\frac{1}{2}})$ and $\w\in D(A^{\frac{1}{2}})$,
\begin{equation}\label{BL4L2L4}
\langle  B(\u, \v), \w\rangle_{D(A^{-\frac{1}{2}}), D(A^{\frac{1}{2}})}\le C  \|\u\|^{\frac{1}{2}} \|A^{\frac{1}{2}} \u\|^{\frac{1}{2}} \|A^{\frac{1}{2}} \v\| \|A^{\frac{1}{2}}\w\|^{\frac{1}{2}} \|\w\|^{\frac{1}{2}}.
\end{equation}
\item For all $ \u\in D(A^{\frac{1}{2}})$, $\v\in D(A)$ and $\w\in\H$,
\begin{equation}\label{BL4L4L2}
(B(\u, \v), \w)\le C  \|\u\|^{\frac{1}{2}} \|A^{\frac{1}{2}} \u\|^{\frac{1}{2}} \|A^{\frac{1}{2}} \v\|^{\frac{1}{2}} \|A\v\|^{\frac{1}{2}} \|\w\|.
\end{equation}
\item For all $\u\in \H$, $\v\in D(A^{\frac{1}{2}})$, and $\w\in D(A)$,
\begin{equation}\label{BL2L2Linf}
\langle B(\u, \v), \w\rangle_{D(A^{-1}), D(A)}\le C \|\u\| \|A^{\frac{1}{2}}\v\|  \|\w\|^{\frac{1}{2}} \|A\w\|^{\frac{1}{2}}.
\end{equation}
\item For all $\u\in D(A)$, $\v \in D(A^{\frac{1}{2}}) $, and $\w\in \H$,
\begin{equation}\label{BLinfL2L2}
(B(\u, \v), \w)\le C \|\u\|^{\frac{1}{2}}\|A\u\|^{\frac{1}{2}}\|A^{\frac{1}{2}}\v\|\|\w\|.
\end{equation}
\item For all $\u\in \H$, $\v\in D(A^{\frac{1}{2}})$, and $\w\in D(A^{\frac{1}{2}})$, 
\begin{equation}\label{Skew_Symmetric_B}
\langle B(\u, \v), \w\rangle_{D(A^{-\frac{1}{2}}), D(A^{\frac{1}{2}})}=-\langle B(\u,\w), \v\rangle_{D(A^{-\frac{1}{2}}), D(A^{\frac{1}{2}})}.
\end{equation}
In particular, 
\begin{equation}\label{Skew_Symmetric_B-bis}
\langle B(\u, \v), \v\rangle_{D(A^{-\frac{1}{2}}), D(A^{\frac{1}{2}})}=0.
\end{equation}
\end{enumerate}
\end{lemma}
\begin{lemma}\label{le:optB}
The bilinear operator $\widetilde B$ is continued as follows. There exists a constant $C>0$ scale invariant such that 
\begin{enumerate}[(i)]
\item For all $\u\in D(A^{\frac{1}{2}})$, $\v\in D(A)$, and  $\w\in\H$,
\begin{equation}\label{tBL4L4L2}
(\widetilde B(\u, \v), \w)\le C \|\u\|^{\frac{1}{2}} \|A^{\frac{1}{2}}\u\|^{\frac{1}{2}}  \|A^{\frac{1}{2}}\v\|^{\frac{1}{2}} \|A\v\|^{\frac{1}{2}} \|\w\|.
\end{equation}
\item For all $ \u\in D(A^{\frac{1}{2}})$, $\v\in D(A^{\frac{1}{2}})$ and $\w\in D(A^{\frac{1}{2}})$,
\begin{equation}\label{tBL4L2L4}
\langle \widetilde B(\u, \v), \w\rangle_{D(A^{-\frac{1}{2}}), D(A^{\frac{1}{2}})}\le C  \|\u\|^{\frac{1}{2}} \|A^{\frac{1}{2}} \u\|^{\frac{1}{2}} \|A^{\frac{1}{2}} \v\| \|A^{\frac{1}{2}}\w\|^{\frac{1}{2}} \|\w\|^{\frac{1}{2}}.
\end{equation}
\item For all $\u\in\H$ and $\v\in D(A^{\frac{1}{2}})$, 
\begin{equation}\label{Skew_Symmetric_tildeB}
 (\widetilde B(\u, \v), \u)=0.
\end{equation}
\end{enumerate}
\end{lemma}
\begin{remark} Gagliardo-Nirenberg's inequality and Agmon's inequality are used to prove the inequalities of Lemmas \ref{le:opB} and \ref{le:optB}. In two-dimensional domains, these inequalities are scaling invariant; therefore, the inequalities of Lemma \ref{le:opB} and \ref{le:optB} inherit the invariance property.   

\end{remark}
\section{Statement of the results}
Here we stay as a reference the hypotheses for $\u_0$ and $\f$ to be used throughout this work.
\begin{enumerate}
\item [(H2)] Assume $\u_0\in D(A)$ and $\f\in L^\infty (0,T; \L^2(\Omega))$ for either $0<T<\infty$ or $T=\infty$. 
\end{enumerate}
Our first step is to modify (\ref{NS-alpha})-(\ref{Helmholtz}) together with (\ref{boundary-condition}) and (\ref{initial-condition}) in order to easily produce an equivalent problem  without pressure. First we  apply the Helmholtz-Leray projector $P$ to (\ref{NS-alpha}) and (\ref{Helmholtz}). Then we obtain the following functional evolution setting 
\begin{equation}\label{NS-alpha-III}
\left\{
\begin{array}{rcl}
\displaystyle
\frac{d\v}{dt}+\nu A\v+ \widetilde B(\u, \v )=P\f,
\\
\u(0)=\u_0,
\end{array}
\right.
\end{equation}
where we have defined $\v=(I+\alpha^2 A)\u$. 
\begin{remark}
Observe that we have switched the role of $\u$ and $\v$ in (\ref{NS-alpha})-(\ref{Helmholtz}) together with (\ref{boundary-condition}) and (\ref{initial-condition}). This might seem a bit strange at first sight. The reasons for this are that we want to keep the notation of the previous papers and to have a unified hypothesis on $\u_0$ being the initial data for (\ref{NS-alpha-III}) and (\ref{NS}) below.          
\end{remark}
Analogously, we apply the Helmholtz-Leray projector to (\ref{NS-PDE}) together with (\ref{boundary-condition}) and (\ref{initial-condition}) to have
\begin{equation}\label{NS}
\left\{
\begin{array}{rcl}
\displaystyle
\frac{d\u}{dt}+\nu A\u+ B(\u, \u )=P\f,
\\
\u(0)=\u_0.
\end{array}
\right.
\end{equation}

Next, we begin by defining the Galerkin approximation to (\ref{NS-alpha-III}) for which we can easily prove existence of solutions and for which we can also show a priori energy estimates that are independent of the regularization parameter $\alpha$. In order to do this, we use the basis of the eigenfunctions $\w_i$, $j\in\mathds{ N}$, for the Stokes operator $A$. For every $n\in\mathds{ N}$, we define the $n${th} Galerkin approximation 
$$\u_n^\alpha=\sum_{i=1}^n a_i^n(t) \w_i $$ satisfying 
\begin{equation}\label{Galerkin}
\left\{
\begin{array}{rcl}
\displaystyle
\frac{d\v_n^\alpha}{dt}+\nu A\v^\alpha_n+P_n \widetilde B(\u_n^\alpha,\v_n^\alpha )&=&P_n\f,
\\
\u_n^\alpha(0)&=&P_n\u_0,
\end{array}
\right.
\end{equation}
where we have defined $\v_n^\alpha= (I+\alpha^2 A)\u_n^\alpha$.

The existence of a solution $\u_n^\alpha$ to (\ref{Galerkin}) on an interval $[0, T_{\alpha, n})$ follows from Carathéodory's theorem. Then a priori estimates show that the solution exists according to the case $t\in[0,T]$ or $[0,+\infty)$. The uniqueness of the solution to (\ref{Galerkin}) is standard; namely, it follows by comparing to different solutions. The smoothness of the solution depends on how smooth is $\f$; in particular, one can prove that $\u_n^\alpha\in H^1(0,T; \V_n)$ under $(\rm H2)$.

Initially, we will derive local-in-time error estimates appropriate on $[0,T]$. Later, we will show how these can be combined to provide  error estimates  that are globally defined on $[0,+\infty)$. 
\begin{theorem}\label{Th1}
Let $T>0$ be fixed. Assume that $(H1)$ and $(H2)$ hold.  Let $\u$ be the solution to (\ref{NS}), and let $\u_n^\alpha$ be the solution to $(\ref{Galerkin})$ on $[0,T]$. 
Then there exists  $K>0$ such that
$$\sup_{0\le t\le T}\left[\|\u_n^\alpha(t)-\u(t)\|^2+\int_0^t \|A^{\frac{1}{2}}(\u_n^\alpha(s)-\u(s))\|^2{\rm d} s\right]\le K\,(\lambda_{1}^{-\frac{1}{2}}\alpha^2+\lambda_{n+1}^{-\frac{3}{2}}),$$
where $K=K(\u_0,\f, \nu, T, \Omega)$, and $\lambda_{n+1}$ is the $(n+1)${th} eigenvalue of the Stokes operator $A$.
\end{theorem}
\begin{theorem}\label{Th2} Let $T=\infty$. Assume that $(H1)$ and $(H2)$ hold. Let $\u$ be the solution to (\ref{NS}), and let $\u_n^\alpha$ be the solution to $(\ref{Galerkin})$ on $[0,+\infty)$. Then there exist  $K_\infty>0$,  $n_0\in\mathds{N}$, and $\alpha_0>0$  such that 
$$\sup_{0\le t <\infty}\|\u_n^\alpha(t)-\u(t)\|^2\le K_\infty\,(\lambda_{1}^{-\frac{1}{2}}\alpha^2+\lambda_{n+1}^{-\frac{3}{2}})$$	
holds provided $n \ge n_0$ and $\alpha\le\alpha_0$,
where $K_\infty=K_\infty(\u_0,\f, \nu, \Omega)$, and $\lambda_{n+1}$ is the $(n+1)$th eigenvalue of the Stokes operator $A$.
\end{theorem}
\section{Local-in-time error estimates}
In this section we will first establish local-in-time a priori energy estimates for the Galerkin approximations $\u^\alpha_n$ to problem (\ref{Galerkin}) independent of the regularization parameter $\alpha$ and the dimension $n$ of $\V_n$. Then, we will be ready to pass to the limit to obtain a strong solution of the Navier-Stokes equations (\ref{NS}), which will inherit the a priori energy estimates from the Galerkin approximations for $\alpha=0$. Finally, we will use both a priori energy estimates to derive local-in-time estimates for the error $\u^\alpha_n-\u$ in the $L^\infty(0,T; \H)$ and $L^2(0,T; D(A^{\frac{1}{2}}))$ norm regarding the regularization parameter $\alpha$ and the eigenvalues $\lambda_{n+1}$ of the Stokes operator $A$.    

\subsection {Local a priori energy estimates}

\begin{lemma}[First energy estimates for $\u^\alpha_n$] Let  $T>0$ be fixed. There exists a constant $E_1=E_1(\u_0,\f, \nu, T, \Omega, \alpha)$ such that the Galerkin approximation $\u^\alpha_n$ defined by problem (\ref{Galerkin}) satisfies
\begin{equation}\label{first-energy}
\sup_{0\le t\le T}\left[ \|\u^\alpha_n(t)\|^2+\alpha^2\|A^{\frac{1}{2}}\u^\alpha_n(t)\|^2+\nu \int_0^t (\| A^{\frac{1}{2}}\u_n^\alpha(s)\|^2+\alpha^2\|A\u^\alpha_n(s)\|^2) {\rm d} s\right]\le E_1.
\end{equation}
\end{lemma}
\begin{proof} Take the $\L^2(\Omega)$-inner product of $(\ref{Galerkin})_1$ with $\u_n^\alpha$ to get
$$
\frac{1}{2}\frac{d }{d t}(\| \u_n^\alpha\|^2 + \alpha^2 \|A^{\frac{1}{2}}\u_n^\alpha\|^2)+\nu (\|A^{\frac{1}{2}} \u_n^\alpha\|^2+\alpha^2 \|A\u^\alpha_n\|^2)  =  (\f,\u_n^\alpha),
$$
where we have used (\ref{Skew_Symmetric_tildeB}). Thus, applying Schwarz'  inequality,  Poincaré's inequality (\ref{Poincare}) and Young's inequality subsequencently to $(\f, \u^\alpha_n)$, one accomplishes 
\begin{equation}\label{lm10-lab1}
\frac{1}{2}\frac{d }{d t}(\| \u_n^\alpha\|^2 + \alpha^2 \|A^{\frac{1}{2}}\u_n^\alpha\|^2)+\nu(\|A^{\frac{1}{2}} \u_n^\alpha\|^2 +\alpha^2 \|A\u^\alpha_n\|^2) \le \frac{1}{2 \nu \lambda_1 } \|\f\|^2 + \frac{\nu}{2} \| A^{\frac{1}{2}}\u_n^\alpha\|^2.
\end{equation}
Finally, integrating over $(0,t)$, for any $t\in[0,T]$, one obtains
$$
\begin{array}{c}
\displaystyle
\| \u_n^\alpha(t)\|^2 + \alpha^2 \|A^{\frac{1}{2}}\u_n^\alpha(t)\|^2+\nu\int_0^t(\|A^{\frac{1}{2}} \u_n^\alpha(s)\|^2 +\alpha^2 \|A\u^\alpha_n(s)\|^2){\rm d} s
\\
\displaystyle
\le\|\u_0\|^2+\alpha^2\|A^{\frac{1}{2}}\u_0\|^2+\frac{1}{\nu \lambda_1}\int_0^T \|\f(s)\|^2 {\rm d} s:=E_1.
\end{array}
$$
\end{proof} 
\begin{lemma}[Second energy estimates for $\u^\alpha_n$] Let $T>0$ be fixed.  There exists a positive constant $E_2=E_2(\u_0,\f, \nu, T, \Omega,\alpha)$ such that  the Galerkin approximation $\u^\alpha_n$ defined by problem (\ref{Galerkin}) satisfies
\begin{equation}\label{second-energy}
\sup_{0\le t\le T}\left[ \|A^{\frac{1}{2}}\u_n^\alpha(t)\|^2 + \alpha^2 \|A\u_n^\alpha(t)\|^2+\nu \int_0^t (\|A\u_n^\alpha(s)\|^2+\alpha^2\|A^{\frac{3}{2}}\u_n^\alpha(s)\|^2) {\rm d} s\right]\le E_2.
\end{equation} 
\end{lemma}
\begin{proof} Take the $\L^2(\Omega)$ inner product of $(\ref{Galerkin})_1$ with $A\u^\alpha_n$ to obtain
\begin{equation}\label{lm11-lab1}
\begin{array}{c}
\displaystyle
\frac{1}{2}\frac{d }{d t}(\|A^{1/2}\u_n^\alpha\|^2+\alpha^2\|A\u^\alpha_n\|^2) + \nu(\|A\u^\alpha_n\|^2+\alpha^2\|A^{\frac{3}{2}}\u^\alpha_n\|^3)
\\
=(\f, A\u^\alpha_n)-(\widetilde B(\u_n^\alpha, \v^\alpha_n), A\u^\alpha_n).
\end{array}
\end{equation}
We shall begin by estimating the term $(\f, A\u^\alpha_n)$. Thus, by Schwarz' and Young's inequality, we have
$$
(\f, A\u^\alpha_n)\le \|\f\|\|A\u^\alpha_n\| \le \frac{C}{\nu} \|\f\|^2+\frac{\nu}{6} \|A\u^\alpha_n\|^2.
$$
Now the relation $\v_n^\alpha=\u_n^\alpha+\alpha^2 A\u^\alpha_n$ allows us to write the term $(\widetilde B(\u_n^\alpha, \v^\alpha_n), A\u^\alpha_n  )$ as:    
$$
\begin{array}{rcl}
(\widetilde B(\u^\alpha_n, \v^\alpha_n), A\u^\alpha_n)&=&(\widetilde B(\u^\alpha_n, \u^\alpha_n), A\u^\alpha_n)+\alpha^2( \widetilde B(\u^\alpha_n, A\u^\alpha_n), A \u^\alpha_n)
\\
&:=& D_1+D_2.
\end{array}
$$
We now combine estimate (\ref{tBL4L4L2})  with Young's inequality to yield   
$$
\begin{array}{rcl}
D_1 
&\le&\displaystyle \frac{C }{\nu^3}\|\u^\alpha_n\|^2 \|A^{\frac{1}{2}}\u^\alpha_n\|^4+\frac{\nu}{6} \|A\u^\alpha_n\|^2
\\
&\le&\displaystyle\frac{C}{\nu^3}E_1\|A^{\frac{1}{2}}\u^\alpha_n\|^2(\|A^{\frac{1}{2}}\u^\alpha_n\|^2+\alpha^2\|A\u^\alpha_n\|^2)
\\
&&\displaystyle+\frac{\nu}{6}( \|A\u^\alpha_n\|^2+\alpha^2\|A^{\frac{3}{2}}\u^\alpha_n\|^2).
\end{array}
$$ 
In a similar fashion, but using estimate (\ref{tBL4L2L4}), it follows the estimate for $D_2$:
$$
\begin{array}{rcl}
D_2
&\le&\displaystyle \frac{C \alpha^2}{\nu^3}\|\u^\alpha_n\|^2\|A^{\frac{1}{2}}\u^\alpha_n\|^2 \|A\u^\alpha_n\|^2+\frac{\nu}{6}\alpha^2 \|A^{\frac{3}{2}}\u^\alpha_n\|^2
\\
&\le&\displaystyle\frac{C}{\nu^3}E_1\|A^{\frac{1}{2}}\u^\alpha_n\|^2(\|A^{\frac{1}{2}}\u^\alpha_n\|^2+\alpha^2\|A\u^\alpha_n\|^2)
\\
&&\displaystyle+\frac{\nu}{6}( \|A\u^\alpha_n\|^2+\alpha^2\|A^{\frac{3}{2}}\u^\alpha_n\|^2).
\end{array} 
$$
Putting all this together into (\ref{lm11-lab1}) gives
\begin{equation}\label{lm11-lab2}
\begin{array}{c}
\displaystyle
\frac{d }{d t}(\|A^{\frac{1}{2}}\u_n^\alpha\|^2+\alpha^2\|A\u^\alpha_n\|^2)+ \nu(\|A\u^\alpha_n\|^2+\alpha^2\|A^{\frac{3}{2}}\u^\alpha_n\|^2)
\\
\displaystyle\le \frac{C}{\nu^3} E_1\|A^{\frac{1}{2}}\u^\alpha_n\|^2(\|A^{\frac{1}{2}}\u^\alpha_n\|^2+\alpha^2\|A\u^\alpha_n\|^2)+\frac{C}{\nu}\|\f\|^2.
\end{array}
\end{equation}
Finally, Grönwall's inequality leads to 
$$
\begin{array}{c}
\displaystyle
\|A^{\frac{1}{2}}\u_n^\alpha(t)\|^2+\alpha^2\|A\u^\alpha_n(t)\|^2 +\nu \int_0^t (\|A\u^\alpha_n(s)\|^2+\alpha^2\|A^\frac{3}{2}\u^\alpha_n(s)\|^2){\rm d} s
\\
\displaystyle\le e^{\frac{C}{\nu^4}E_1^2}\left\{\|A^{\frac{1}{2}}\u_0\|^2+\alpha^2\|A\u_0\|^2+ \frac{C}{\nu} \int_0^T\|\f(s)\|^2{\rm d} s\right\}:=E_2,
\end{array}
$$
for all $t\in [0,T]$.
\end{proof}

\begin{lemma} Let $T>0$ be fixed. There exists a positive constant $E_3=E_3(\u_0,\f, \nu, T, \Omega,\alpha)$ such that  the Galerkin approximation $\u^\alpha_n$ defined by problem (\ref{Galerkin}) satisfies
\begin{equation}\label{est:derivative-u}
\int_0^T \|\frac{d}{dt}\u_n^\alpha(t)\|^2 d{\rm t}\le E_3.
\end{equation}
\end{lemma}
\begin{proof} Applying the operator $(I+\alpha^2 A)^{-1}$ to $(\ref{Galerkin})_1$, we write 
$$
\begin{array}{rcl}
\displaystyle
\frac{d\u_n^\alpha}{dt}&=&-\nu A\u^\alpha_n-(I+\alpha^2 A)^{-1} P_n \widetilde B(\u_n^\alpha, \v_n^\alpha )
\\
&&+(I+\alpha^2 A)^{-1}P_n\f.
\end{array}
$$
Thus, we have
$$
\begin{array}{rcl}
\displaystyle
\|\frac{d\u_n^\alpha}{dt}\|^2&\le& C \nu^2 \|A\u^\alpha_n\|^2+ C \|(I+\alpha^2 A)^{-1} P_n \widetilde B(\u_n^\alpha, \v_n^\alpha )\|^2
\\
&&+ C\|(I+\alpha^2 A)^{-1}P_n\f\|^2.
\end{array}
$$
 
It is clear from (\ref{second-energy}) that 
$$\nu^2\int_0^T\|A\u^\alpha_n(s)\|^2 d{\rm s}\le \nu E_2$$
From (\ref{ResolventL2}) and (\ref{stabL2}), we have 
$$
\begin{array}{rcl}
\|(I+\alpha^2 A)^{-1}P_n\widetilde B(\u^\alpha_n,\v^\alpha_n)\|^2&\le& \|P_n\widetilde B(\u^\alpha_n,\v^\alpha_n)\|^2\le \|\u^\alpha_n\times(\nabla\times\v^\alpha_n))\|^2
\\
&\le& \|\u_n^\alpha\|_{L^4(\Omega)}^2 \|A^{\frac{1}{2}}\v^\alpha_n\|_{L^4(\Omega)}^2\le C \|\u^\alpha_n\| \|A^{\frac{1}{2}}\u^\alpha_n\| \|A^{\frac{1}{2}}\v^\alpha_n\| \|A\v^\alpha_n\|
\\
&\le& \|\u^\alpha_n\| \|A^{\frac{1}{2}}\u^\alpha_n\| (\|A^{\frac{1}{2}}\u^\alpha_n\|+\alpha\|A\u^\alpha_n\|) (\|A\u^\alpha_n\|+\alpha\|A^{\frac{3}{2}}\u^\alpha_n\|).
\end{array} 
$$
Using Schwarz' inequality and integrating over $[0,T]$ gives 
$$
\begin{array}{rcl}
\displaystyle
\int_0^T\|P_n\widetilde B(\u^\alpha_n(s),\v^\alpha_n(s))\|^2\,d{\rm s}&\le& 
\displaystyle
\frac{1}{\nu^2}\int_0^T \|\u^\alpha_n(s)\|^2 \|A^{\frac{1}{2}}\u^\alpha_n(s)\|^2 (\|A^{\frac{1}{2}}\u^\alpha_n(s)\|^2+\alpha^2\|A\u^\alpha_n(s)\|^2)\,d{\rm s}
\\
&&
\displaystyle
+\nu^2 \int_0^T (\|A\u^\alpha_n(s)\|^2+\alpha^2\|A^{\frac{3}{2}}\u^\alpha_n(s)\|^2)d{\rm s}
\\
&\le&\displaystyle E_2(\frac{1}{\nu^2} E_1 E_2 T+ \nu).
\end{array} 
$$
Moreover, we have
$$\int_0^T\|(I+\alpha^2 A)^{-1}P_n\f(s)\| d{\rm s}\le\int_0^T\|\f(s)\|^2\,d{\rm s}.$$
Therefore,
$$
\int_0^T\|\frac{d}{dt}\u^\alpha_n(s)\|^2 d{\rm s}\le E_2(\frac{1}{\nu^2}E_1 E_2 T+ 2\nu)+\|\f\|^2_{L^2(0,T;L^2(\Omega))}:=E_3.
$$

\end{proof}

The bound (\ref{second-energy}) on the sequence $\{\u^\alpha_n\}_{\alpha,n}$ allows to prove that there exist a subsequence $\{\u^{\alpha_j}_{n_j}\}$ and a function $\u$ such that   
$$
\begin{array}{rcl}
\u^{\alpha_j}_{n_j}\to\u &\hbox{ weakly-$\star$ in }&  L^\infty(0,T; D(A^{\frac{1}{2}})),  
\\
\u^{\alpha_j}_{n_j} \to \u  &\hbox{ weakly in }&L^2(0,T; D(A)),
\end{array}
$$
and, by a compactness result of the Aubin-Lions type together with (\ref{est:derivative-u}), such that
$$
\u^{\alpha_j}_{n_j} \to \u\quad \hbox{ strongly in }\quad  L^2(0,T; D(A^{\frac{1}{2}})),
$$
with $(\alpha_j, n_j)\to (0,\infty)$ as $j\to\infty$, where $\u$ is a strong solution of the Navier-Stokes equations. The passage to the limit is routine. This convergence is discussed in detail by Foias et al. in \cite{Foias-Holm-Titi} for weak solutions.  

The strong solution $\u$ to the Navier-Stokes equations (\ref{NS}) inherits the bounds (\ref{first-energy}) and (\ref{second-energy}) for $\alpha=0$ due to the lower semi-continuity of  the $L^\infty(0,T; \H)$  and $L^2(0,T; D(A^{\frac{1}{2}}))$ norm.
 
\begin{theorem} Let $T>0$ be fixed. There exist two positive constants $\widetilde E_1=\widetilde E_1(\u_0,\f, \nu, T, \Omega)$ and  $\widetilde E_2=\widetilde E_2(\u_0,\f, \nu, T, \Omega)$, which are $E_1$ and $E_2$ with $\alpha=0$, respectively, such that the unique solution $\u$ to problem (\ref{NS})  satisfies
$$
\sup_{0\le t\le T} \left[ \|\u(t)\|^2+\nu\int_0^t\|A^{\frac{1}{2}}\u(s)\|^2{\rm d} s\right]\le \widetilde E_1
$$
and
$$
\sup_{0\le t\le T} \left[ \|A^{\frac{1}{2}}\u(t)\|^2+\nu\int_0^t\|A\u(s)\|^2{\rm d} s\right]\le \widetilde E_2.
$$
\end{theorem}

\subsection {Proof of Theorem \ref{Th1}}

We split the error $\u-\u^\alpha_n$ into to two parts, $\e_n =\u-\boldsymbol\eta_n=P^\perp_n\u $, where  $\boldsymbol\eta_n=P_n\u$, and  $\z_n^\alpha =\u_n^\alpha -\boldsymbol\eta_n$. Thus $\u -\u_n^\alpha = \e_n -\z^\alpha_n$. 

The next result concerns the error estimates for $\e_n$. 
\begin{lemma} Let $T>0$ be fixed. There exists a positive constant $K_1=K_1(\u_0,\f, \nu, \Omega)$  such that 
\begin{equation}\label{error-e}
\sup_{0\le t\le T}\|\e_n(t)\|^2\le K_1  \lambda_{n+1}^{-\frac{3}{2}}.
\end{equation}
\end{lemma} 
\begin{proof}
Applying $P^\perp_n$ to (\ref{NS}), we get 
\begin{equation}\label{lm13-lab1}
\frac{d}{dt}\e_n+\nu A\e_n=-P_n^\perp B(\u, \u)+P_n^\perp\f.
\end{equation}
Next, take the $\L^2(\Omega)$-inner product of (\ref{lm13-lab1}) with $\e_n$ to obtain
\begin{equation}\label{lm13-lab2}
\frac{1}{2}\frac{d}{dt} \|\e_n\|^2+\nu\|A^{\frac{1}{2}}\e_n\|^2=-(P^\perp_n B(\u, \u), \e_n)+(\f, \e_n ).
\end{equation}
Let us bound the right-hand side of  (\ref{lm13-lab2}). Making use of (\ref{BL4L2L4}) and (\ref{errorL2H1}), we estimate
$$
\begin{array}{rcl}
(P^\perp_n B(\u, \u), \e_n)&\le& C \|\u\|^{\frac{1}{2}} \|A^{\frac{1}{2}}\u\|^{\frac{3}{2}} \|\e_n\|^{\frac{1}{2}} \|A^{\frac{1}{2}} \e_n\|^{\frac{1}{2}}
\\
&\le& C \|\u\|^{\frac{1}{2}} \|A^{\frac{1}{2}}\u\|^\frac{3}{2} \lambda_{n+1}^{-\frac{1}{4}} \|A^{\frac{1}{2}} \e_n\|
\\
&\le&\displaystyle\frac{C}{\nu}\lambda_{n+1}^{-\frac{1}{2}} \|\u\| \|A^{\frac{1}{2}}\u\|^3+\frac{\nu}{4} \|A^{\frac{1}{2}}\e_n \|^2
\\
&\le&\displaystyle\frac{C}{\nu}\lambda_{1}^{-\frac{1}{2}}\lambda_{n+1}^{-\frac{1}{2}} \widetilde E^2_2+\frac{\nu}{4} \|A^{\frac{1}{2}}\e_n \|^2,
\end{array}
$$
where we have used (\ref{Poincare}) in the last line. Also, 
$$
(\f, \e_n )\le \|\f\| \|\e_n\|\le \lambda^{-\frac{1}{2}}_{n+1} \|\f\| \|A^{\frac{1}{2}}\e_n\|\le\frac{C}{\nu}\lambda^{-1}_{n+1} \|\f\|^2+\frac{\nu}{4}\|A^{\frac{1}{2}}\e_n\|^2.
$$
Thus we achieve the following differential inequality:
$$
\frac{d}{dt} \|\e_n\|^2+\nu\|A^{\frac{1}{2}}\e_n\|^2\le \frac{C}{\nu}\lambda^{-\frac{1}{2}}_1  \lambda^{-\frac{1}{2}}_{n+1}\widetilde E_2^2+\frac{C}{\nu}\lambda^{-1}_{n+1}\|\f\|^2.
$$
Taking advantage of  (\ref{errorL2H1}), we get
$$
\frac{d}{dt} \|\e_n\|^2+\nu\lambda_{n+1}\|\e_n\|^2\le\frac{C}{\nu} \lambda^{-\frac{1}{2}}_1  \lambda^{-\frac{1}{2}}_{n+1}\widetilde E_2^2+\frac{C}{\nu}\lambda^{-1}_{n+1}\|\f\|^2.
$$
Therefore,
$$
\frac{d}{dt}(e^{\nu\lambda_{n+1} t} \|\e_n\|^2)\le\frac{C}{\nu} e^{\nu\lambda_{n+1} t} \lambda^{-\frac{1}{2}}_1  \lambda^{-\frac{1}{2}}_{n+1}\widetilde E_2^2  +\frac{C}{\nu}e^{\nu\lambda_{n+1} t} \lambda^{-1}_{n+1} \|\f\|^2.
$$
Integrating over $(0,t)$, for any $t\in[0,T]$, we find 
$$ 
\begin{array}{rcl}
\|\e_n(t)\|^2 &\le&\displaystyle e^{-\nu\lambda_{n+1} t} \|\e_n(0)\|^2+\frac{C }{\nu}\int_0^t e^{-\nu\lambda_{n+1}(t-s)}(\lambda^{-\frac{1}{2}}_1  \lambda^{-\frac{1}{2}}_{n+1}\widetilde E_2^2+\lambda^{-1}_{n+1}\|\f(s)\|^2){\rm d} s
\\
&\le&\displaystyle \|\e_n(0)\|^2+\frac{C }{\nu^2}\lambda^{-\frac{3}{2}}_{n+1}\left\{\lambda^{-\frac{1}{2}}_1 \widetilde E_2^2+\lambda^{-\frac{1}{2}}_{n+1}\|\f\|^2_{L^\infty(0,T; L^2(\Omega))}\right\}.
\end{array}  
$$
Finally, from (\ref{errorH1H2yL2H2}), we have $\|\e_n(0)\|^2\le C \lambda_{n+1}^{-2} \|A\u_0\|^2$. Hence, we see that
\begin{equation}\label{lm13-lab3}
\begin{array}{rcl}
\|\e_n(t)\|^2&\le&\displaystyle C \left\{\lambda_{n+1}^{-\frac{1}{2}}\|A\u_0\|^2+\frac{1}{\nu^2}(\lambda^{-\frac{1}{2}}_1 \widetilde E_2^2+\lambda^{-\frac{1}{2}}_{n+1}\|\f\|^2_{L^\infty(0,T; L^2(\Omega))})\right\}\lambda_{n+1}^{-\frac{3}{2}}
\\
&:=&
K_1\lambda_{n+1}^{-\frac{3}{2}}.
\end{array}
\end{equation}
\end{proof}

For the error $\z_n^\alpha$, we start by writing its own equation. In order to do this, we first apply the operator $(I+\alpha^2 A)^{-1}$ to $(\ref{Galerkin})_1$ to obtain 
\begin{equation}\label{th1-lab1}
\begin{array}{rcl}
\displaystyle
\frac{d\u_n^\alpha}{dt}+\nu A\u^\alpha_n&=&-(I+\alpha^2 A)^{-1} P_n (B(\u_n^\alpha, \v_n^\alpha )+ B^\star (\u_n^\alpha, \v_n^\alpha) )
\\
&&+(I+\alpha^2 A)^{-1}P_n\f,
\end{array}
\end{equation}
where we have used the relation (\ref{rel:tB-B}).

Next, observe that $\boldsymbol\eta_n=P_n\u$  satisfies  
\begin{equation}\label{th1-lab2}
\frac{d}{dt}\boldsymbol\eta_n+\nu A\boldsymbol\eta_n=-P_nB(\u, \u)+P_n\f.
\end{equation}
This is readily seen by applying the the finite-dimensional Helmholtz-Leray operator $P_n$ to (\ref{NS}). Subtracting (\ref{th1-lab2}) from (\ref{th1-lab1}) gives
\begin{equation}\label{eq:z}
	\begin{array}{rcl}
	\displaystyle
	\frac{d\z^\alpha_n}{dt}+\nu A\z^\alpha_n&=& P_n B(\u, \u)-(I+\alpha^2 A)^{-1} P_n  B (\u^\alpha_n, \v^\alpha_n)
\\
	&&+(I+\alpha^2 A)^{-1} P_nB^\star(\u_n^\alpha,\v^\alpha_n)
\\
	&&+((I+\alpha^2 A)^{-1}-I)P_n\f.
	\end{array}
\end{equation}
Splitting the right-hand side of (\ref{eq:z})  appropriately as 
$$
\begin{array}{rcl}
P_n B(\u, \u)&=&P_n B(\u, \u)\pm P_n(\u_n^\alpha, \u_n^\alpha)
\\
&=&P_n B(\e_n-\z_n^\alpha, \u)+P_n B(\u_n^\alpha, \e_n-\z_n^\alpha)+P_n B(\u_n^\alpha, \u_n^\alpha)
\\
&=&-P_nB(\u, \z^\alpha_n)-P_nB(\z^\alpha_n, \u)-P_nB(\z^\alpha_n,\z^\alpha_n)
\\
&&+P_nB(\z^\alpha_n, \e_n)+P_nB(\e_n,\z^\alpha_n)+P_nB(\u, \e_n)
\\
&& +P_nB(\e_n, \boldsymbol\eta_n)+P_n B(\u_n^\alpha, \u_n^\alpha)
,
\end{array}
$$
we obtain 
\begin{equation}\label{eq:z_II}
\begin{array}{rcl}
\displaystyle
\frac{d\z^\alpha_n}{dt}+\nu A\z^\alpha_n&=&-P_nB(\u, \z^\alpha_n)-P_nB(\z^\alpha_n, \u)-P_nB(\z^\alpha_n,\z^\alpha_n)
\\
&&+P_nB(\z^\alpha_n, \e_n)+P_nB(\e_n,\z^\alpha_n)
\\
&&+P_nB(\u, \e_n) +P_nB(\e_n, \boldsymbol\eta_n)
\\
&&+(I+\alpha^2 A )^{-1}P_n (B(\u^\alpha_n, \u^\alpha_n)-B(\u^\alpha_n,\v^\alpha_n))
\\
&&-((I+\alpha^2 A)^{-1}-I) P_nB(\u^\alpha_n, \u^\alpha_n)
\\
&&+(I+\alpha^2 A)^{-1} P_nB^\star(\u_n^\alpha, \v^\alpha_n)
\\
&&+((I+\alpha^2 A)^{-1}-I)P_n\f.
\end{array}
\end{equation}

Now we are prepared to prove the local-in-time error estimate announced in Theorem \ref{Th1}. 
Taking the $\L^2(\Omega)$ inner product of (\ref{eq:z_II}) with $\z_n^\alpha$, we get
\begin{equation}\label{th1-lab3}
\begin{array}{rcl}
\displaystyle
\frac{1}{2}\frac{d}{dt} \|\z_n^\alpha\|^2&+&\nu\|A^{\frac{1}{2}} \z_n^\alpha\|^2
= -(B(\u, \z^\alpha_n), \z^\alpha_n) 
\\
&-&(B(\z_n^\alpha, \u), \z^\alpha_n)-(B(\z^\alpha_n, \z^\alpha_h), \z^\alpha_n)+(B(\z^\alpha_n, \e_n), \z^\alpha_n)
\\
&-&( B(\e_n, \z^\alpha_n), \z_n^\alpha)+(B(\u, \e_n), \z^\alpha_n)-( B(\e_n, \boldsymbol\eta_n), \z_n^\alpha)
\\
&+&((I+\alpha^2 A )^{-1}P_n (B(\u^\alpha_n, \u^\alpha_n)-B(\u^\alpha_n,\v^\alpha_n)), \z_n^\alpha)
\\
&-&((I+\alpha^2 A)^{-1}-I) P_n B(\u^\alpha_n, \u^\alpha_n), \z^\alpha_n)
\\
&+&((I+\alpha^2 A)^{-1} P_n B^\star(\u_n^\alpha, (I+\alpha^2 A)\u^\alpha_n), \z^\alpha_n)
\\
&+&(((I+\alpha^2 A)^{-1}-I)P_n\f, \z^\alpha_n)
\\
&:=&\displaystyle\sum_{i=1}^{11}J_i.
\end{array}
\end{equation}

The right-hand side of (\ref{th1-lab3}) will be handled separately. It is clear that $J_i=0$, for $i=1, 3, 5$, from (\ref{Skew_Symmetric_B-bis}). Let $\varepsilon$ be a positive constant (to be adjusted below). The skew-symmetric propierty  (\ref{Skew_Symmetric_B}) of $B$ combined with (\ref{BL2L2Linf}) and (\ref{Poincare}) gives   
$$
\begin{array}{rcl}
J_2&=&(B(\z_n^\alpha, \z^\alpha_n), \u )\le C \|\z^\alpha_n\| \|A^{\frac{1}{2}}\z_n^\alpha\| \|\u\|^{\frac{1}{2}} \|A\u\|^{\frac{1}{2}}
\\
&\le&\displaystyle \frac{C_\varepsilon}{\nu} \|\u\| \|A\u\|  \|\z^\alpha_n\|^2+\nu \varepsilon \|A^{\frac{1}{2}}\z_n^\alpha\|^2
\\
&\le&\displaystyle \frac{C_\varepsilon}{\nu} \lambda_1^{-1}\|A\u\|^2  \|\z^\alpha_n\|^2+\nu \varepsilon \|A^{\frac{1}{2}}\z_n^\alpha\|^2
\end{array}
$$ 
and
$$
\begin{array}{rcl}
J_4&=&-(B(\z_n^\alpha, \z^\alpha_n), \e_n )\le \|\z^\alpha_n\| \|A^{\frac{1}{2}}\z_n^\alpha\| \|\e_n\|^{\frac{1}{2}} \|A\e_n\|^{\frac{1}{2}}
\\
&\le&\displaystyle \frac{C_\varepsilon}{\nu}  \|\u\| \|A\u\|  \|\z^\alpha_n\|^2+\nu \varepsilon \|A^{\frac{1}{2}}\z_n^\alpha\|^2
\\
&\le&\displaystyle \frac{C_\varepsilon}{\nu} \lambda_1^{-1}\|A\u\|^2  \|\z^\alpha_n\|^2+\nu \varepsilon \|A^{\frac{1}{2}}\z_n^\alpha\|^2,
\end{array}
$$
where we have also used (\ref{stabL2}) and (\ref{stabH2}) for bounding $\|\e_n\|\le 2 \|\u\|$ and $ \|A\e_n\|\le 2 \|A\u\| $ in $J_4$. Now, combing successively  (\ref{Skew_Symmetric_B}), (\ref{BL4L2L4}),   (\ref{errorH1H2yL2H2}), (\ref{stabH2}), and  (\ref{Poincare}), we get 
$$
\begin{array}{rcl}
J_6&=&-(B(\u,\z_n^\alpha), \e_n)\le C \|\u\|^{\frac{1}{2}}  \|A^{\frac{1}{2}}\u\|^{\frac{1}{2}}   \|A^{\frac{1}{2}}\z_n^\alpha\| \|\e_n\|^{\frac{1}{2}} \|A^{\frac{1}{2}}\e_n\|^{\frac{1}{2}} 
\\
&\le& C \|\u\|^{\frac{1}{2}}  \|A^{\frac{1}{2}}\u\|^{\frac{1}{2}}  \|A^{\frac{1}{2}}\z_n^\alpha\| \lambda^{-\frac{3}{4}}_{n+1} \|A\u\| 
\\
&\le& C \lambda_1^{-\frac{1}{4}}  \|A^{\frac{1}{2}}\u\|  \|A^{\frac{1}{2}}\z_n^\alpha\| \lambda^{-\frac{3}{4}}_{n+1} \|A\u\| 
\\
&\le&\displaystyle \frac{C_\varepsilon}{\nu}  \widetilde E_2 \lambda_1^{-\frac{1}{2}}\lambda^{-\frac{3}{2}}_{n+1}  \|A\u\|^2 +\varepsilon \nu\|A^{\frac{1}{2}}\z_n^\alpha\|^2.
\end{array}
$$
Analogous to $J_6$, we have that $J_7$ can be estimated as:   
$$
\begin{array}{rcl}
J_7&=&-( B(\e_n, \z_n^\alpha), \boldsymbol\eta_n)\le C \|\e_n\|^{\frac{1}{2}} \|A^\frac{1}{2}\e_n\|^{\frac{1}{2}} \|A^{\frac{1}{2}}\z^\alpha_n\| \|\boldsymbol\eta_n\|^{\frac{1}{2}} \|A^{\frac{1}{2}}\boldsymbol\eta_n\|^{\frac{1}{2}}
\\
&\le&\displaystyle \frac{C_\varepsilon}{\nu}  \widetilde E_2 \lambda_1^{-\frac{1}{2}}\lambda^{-\frac{3}{2}}_{n+1}  \|A\u\|^2 +\varepsilon \nu\|A^{\frac{1}{2}}\z_n^\alpha\|^2,
\end{array}
$$
where we have also used (\ref{stabL2}) and (\ref{stabH1}) for bounding $\|\boldsymbol\eta_n\|\le  \|\u\|$ and $ \|A^{\frac{1}{2}}\boldsymbol\eta_n\|\le  \|A^{\frac{1}{2}}\u\| $.
From the fact that $(I+\alpha^2 A)^{-1}$ is a self-adjoint operator and in view of the definition of $\v^\alpha_n=(I+\alpha^2 A) \u^\alpha_n$, we write
$$
\begin{array}{c}
J_8
=-\alpha^2(B(\u^\alpha_n,A\u_n^\alpha), (I+\alpha^2 A)^{-1} \z_n^\alpha)
\\
=\alpha^2 (B(\u_n^\alpha, (I+\alpha^2 A)^{-1}\z^\alpha_n), A\u^\alpha_n),
\end{array}
$$
where in the last line we have utilized (\ref{Skew_Symmetric_B}).
Thus, in virtue of (\ref{BLinfL2L2}), (\ref{ResolventH1}), and (\ref{Poincare}), we get 
$$
\begin{array}{rcl}
J_8&\le& C  \alpha^2  \|\u^\alpha_n\|^{\frac{1}{2}} \|A\u^\alpha_n\|^{\frac{1}{2}} \| A^{\frac{1}{2}} (I+\alpha^2 A)^{-1}\z^\alpha_n\|\|A\u^\alpha_n\|
\\
&\le& C  \alpha^2  \|\u^\alpha_n\|^{\frac{1}{2}} \|A\u^\alpha_n\|^{\frac{3}{2}} \|A^{\frac{1}{2}}\z^\alpha_n\|
\\
&\le&\displaystyle \frac{C_\varepsilon}{\nu} E_2 \lambda_1^{-\frac{1}{2}} \alpha^3   \|A\u^\alpha_n\|^2+\varepsilon\nu \|A^{\frac{1}{2}}\z^\alpha_n\|^2.
\end{array}
$$
In order to estimate $J_9$, we use identity (\ref{I-Ja=aAJa}) to obtain 
$$
\begin{array}{rcl}
J_9&=&\alpha^2 (A (I+\alpha^2 A)^{-1} P_n B(\u^\alpha_n, \u^\alpha_n), \z^\alpha_n)
\\
&=&\alpha^2 ( B(\u^\alpha_n, \u^\alpha_n), P_n  (I+\alpha^2 A)^{-1} A \z^\alpha_n)
\\
&=&\alpha^2 (B(\u^\alpha_n, \u^\alpha_n),  (I+\alpha^2 A)^{-1} A \z^\alpha_n)
\\
&=&\alpha ((\alpha^2 A)^{\frac{1}{2}}(I+\alpha^2 A)^{-1} B(\u^\alpha_n, \u^\alpha_n),  A^{\frac{1}{2}}\z^\alpha_n).
\end{array}
$$
Observe that we have applied that $A(I+\alpha^2 A)^{-1}$ is an adjoint operator and  neglected $P_n$ since $(I+\alpha^2 A)^{-1} A\z_n^\alpha$ belongs to $\V_n$.
Now, from (\ref{Resolvent-alphaH1}) and (\ref{Poincare}), we have 
$$
\begin{array}{rcl}
J_9&\le&\alpha\|(\alpha^2 A)^{\frac{1}{2}}(I+\alpha A)^{-1}B(\u^\alpha_n, \u^\alpha_n)\| \|A^{\frac{1}{2}} \z^\alpha_n\| 
\\
&\le&\alpha \|B(\u^\alpha_n, \u^\alpha_n)\|  \|A^{\frac{1}{2}} \z^\alpha_n\|
\\
&\le&\alpha \|\u^\alpha_n\|_{L^\infty(\Omega)} \|A^{\frac{1}{2}}\u^\alpha_n\|  \|A^{\frac{1}{2}} \z^\alpha_n\|
\\
&\le&\alpha \|\u^\alpha_n\|^{\frac{1}{2}} \|A\u^\alpha_n\|^\frac{1}{2}\|A^\frac{1}{2}\u^\alpha_n\| \|A^{\frac{1}{2}}\z^\alpha_n\|
\\
&\le& \displaystyle  \frac{C_\varepsilon}{\nu} E_2 \lambda_1^{-1} \alpha^2 \|A\u^\alpha_n\|^2 +\varepsilon\nu\|A^{\frac{1}{2}}\z^\alpha_n\|^2.
\end{array}
$$
It follows from (\ref{rel:B_starB}) and (\ref{Skew_Symmetric_B-bis})  that 
$$
\begin{array}{rcl}
J_{10}&=&( (B^\star(\u_n^\alpha, (I+\alpha^2 A)\u^\alpha_n))- B^\star(\u_n^\alpha, \u^\alpha_n)), (I+\alpha^2 A)^{-1} \z^\alpha_n)
\\
&=& \alpha^2 (  B^\star(\u_n^\alpha,  A\u^\alpha_n), (I+\alpha^2 A)^{-1} \z^\alpha_n)
\\
&=&\alpha^2 ( B((I+\alpha^2 A)^{-1}\z^\alpha_n,  A\u^\alpha_n), \u_n^\alpha)
\\
&=&-\alpha^2 ( B((I+\alpha^2 A)^{-1} \z_n^\alpha, \u^\alpha_n),  A\u^\alpha_n).
\end{array}
$$
Next, thanks to (\ref{BL4L4L2}), (\ref{ResolventL2}),  (\ref{ResolventH1}) and (\ref{Poincare}), we find that 
$$
\begin{array}{rcl}
J_{10}&\le& \alpha^2 \|(I+\alpha^2 A)^{-1}\z_n^\alpha\|^{\frac{1}{2}} \|A^{\frac{1}{2}}(I+\alpha^2 A)^{-1}\z_n^\alpha\|^{\frac{1}{2}} \|A^{\frac{1}{2}}\u^\alpha_n\|^{\frac{1}{2}} \|A\u^\alpha_n\|^{\frac{3}{2}}
\\
&\le& \displaystyle \frac{C_\varepsilon}{\nu} E_2 \lambda_1^{-\frac{1}{2}} \alpha^3 \|A\u^\alpha_n\|^2+\varepsilon \nu\|A^{\frac{1}{2}} \z_n^\alpha\|^2.
\end{array}
$$
It is readily to bound $J_{11}$ as 
$$
\begin{array}{rcl}
J_{11}&=&\alpha^2(A(I+\alpha^2 A)^{-1} P_n\f, \z^\alpha_n)=\alpha((\alpha^2A)^{\frac{1}{2}}(I+\alpha^2 A)^{-1} P_n\f, A^{\frac{1}{2}} \z^\alpha_n)
\\
&\le &\alpha\|P_n\f\| \| A^{\frac{1}{2}} \z^\alpha_n\|\le\displaystyle\frac{C_\varepsilon}{\nu} \alpha^2 \|\f\|^2+\varepsilon\nu \|A^{\frac{1}{2}}\z^\alpha_n\|^2.
\end{array}
$$
%
Collecting all the above estimates and choosing $\varepsilon$ appropriately, we have
$$
\begin{array}{rcl}
\displaystyle
\frac{d}{dt} \|\z_n^\alpha\|^2+\nu\|A^{\frac{1}{2}} \z_n^\alpha\|^2&\le&
\displaystyle\frac{C}{\nu} \lambda_1^{-1}\|A\u\|^2  \|\z^\alpha_n\|^2+\frac{C}{\nu}  \widetilde E_2 \lambda_1^{-\frac{1}{2}}\lambda^{-\frac{3}{2}}_{n+1}  \|A\u\|^2
\\
&&\displaystyle + \frac{C}{\nu} E_2 \lambda_1^{-\frac{1}{2}} \alpha^3   \|A\u^\alpha_n\|^2+ \frac{C}{\nu} E_2 \lambda_1^{-1} \alpha^2 \|A\u^\alpha_n\|^2
\\
&&\displaystyle +\frac{C}{\nu} \alpha^2 \|\f\|^2.
\end{array}
$$
Equivalently, 
\begin{equation}\label{th1-lab4}
\begin{array}{rcl}
\displaystyle
\frac{d}{dt} \|\z_n^\alpha\|^2+\nu\|A^{\frac{1}{2}} \z_n^\alpha\|^2&\le& \displaystyle\frac{C}{\nu} \lambda_1^{-1}\|A\u\|^2  \|\z^\alpha_n\|^2+\frac{C}{\nu} (\lambda_1^{-\frac{1}{2}}\alpha^2+\lambda_{n+1}^{-\frac{3}{2}}) \left[  \lambda_1^{-\frac{1}{2}}  \widetilde E_{2}  \|A\u\|^2\right.
\\
&&\left.+E_{2}  (\lambda_1^{-\frac{1}{2}}+\alpha)\|A \u^\alpha_n\|^2+\lambda_1^{\frac{1}{2}}\|\f\|^2\right]
\end{array}
\end{equation}


Applying Grönwall's inequality  yields

\begin{align*}
\|\z^\alpha_n(t)\|^2+&\nu \int_0^t\|A^{\frac{1}{2}} \z^\alpha_n(s)\|^2{\rm d} s \le \frac{C}{\nu} e^{\frac{C}{\nu^2} \lambda_1^{-1} E_2} (\lambda_1^{-\frac{1}{2}}\alpha^2+\lambda_{n+1}^{-\frac{3}{2}})\times
\\
&\times\left[\lambda_1^{-\frac{1}{2}}  \widetilde E_{2}  \int_0^T\|A\u(s)\|^2 {\rm d}s+ E_2(\lambda_1^{-\frac{1}{2}}+\alpha)\int_0^T\|A \u^\alpha_n(s)\|^2{\rm d} s +\lambda_1^{\frac{1}{2}}\int_0^T\|\f(s)\|^2 {\rm d}s\right]  
\\
&\displaystyle
:=K_2(\lambda^{-\frac{1}{2}}_1\alpha^2+\lambda_{n+1}^{-\frac{3}{2}}) ,
\end{align*}
where we have used the fact that $\z_n^\alpha(0)=\boldsymbol{0}$. To conclude the proof of Theorem \ref{Th1}, we combine the above estimate and (\ref{error-e}) with the triangle inequality and choose $K=\max\{K_1, K_2\}$.
\section{Global-in-time error estimates}

Without further assumptions on the solution $\u$ to the Navier-Stokes equations (\ref{NS}), global-in-time error estimates cannot be asserted. Therefore, to go further, we need to introduce the concept of the $\L^2(\Omega)$ stability for solutions of the Navier-Stokes equations. This stability condition deals with the behavior of perturbations of $\u$; namely, the difference between neighboring solutions must decay as time goes to infinity. Once we know that the solution $\u$ is stable in the sense of the $\L^2(\Omega)$ norm, we will be able to obtain global-in-time estimates for the error $\u-\u^\alpha_n$ in the $L^\infty(0,\infty; \H)$ norm concerning the regularization parameter $\alpha$ and the eigenvalue $\lambda_{n+1}$ of the Stokes operator $A$. In doing so, we will first prove global-in-time a priori energy estimates. 

\subsection{Global a priori energy estimates}
\begin{lemma}[First energy estimates for $\u^\alpha_n$] Let $T=\infty$. There exists a positive constant $E_{1,\infty}=E_{1,\infty}(\u_0,\f, \nu, T, \Omega,\alpha)$ such that the Galerkin approximation $\u^\alpha_n$ defined by problem (\ref{Galerkin}) satisfies
\begin{equation}\label{first-energy-infty}
\sup_{0\le t < \infty}\left[ \|\u^\alpha_n(t)\|^2+\alpha^2\|A^{\frac{1}{2}}\u^\alpha_n(t)\|^2\right]\le E_{1,\infty}.
\end{equation}
Furthermore, we have, for $0\le t_0\le t$, 
\begin{equation}\label{first-energy-infty2}
\nu\int_{t_0}^t(\|A^{\frac{1}{2}}\u^\alpha_n(s)\|^2+\alpha^2\|A\u^\alpha_n(s)\|^2){\rm d} s\le E_{1,\infty}(1+\nu\lambda_1(t-t_0)).
\end{equation}
\end{lemma}
\begin{proof} To start with, we take advantage of (\ref{lm10-lab1})  to get 
\begin{equation}\label{lm14-lab1}
\frac{d }{d t}(\| \u_n^\alpha\|^2 + \alpha^2 \|A^{\frac{1}{2}}\u_n^\alpha\|^2)+\nu (\|A^{\frac{1}{2}} \u_n^\alpha\|^2+\alpha^2\|A\u^\alpha_n\|^2)  \le \frac{1}{\nu \lambda_1 } \|\f\|^2.
\end{equation}
By Poincaré's inequality (\ref{Poincare}), we find that
$$
\frac{d }{d t}(\| \u_n^\alpha\|^2 + \alpha^2 \|\u_n^\alpha\|^2)+\nu \lambda_1( \|\u_n^\alpha\|^2 +\|A^{\frac{1}{2}}\u^\alpha_n\|^2) \le \frac{1}{ \nu \lambda_1 } \|\f\|^2.
$$
Multiplying by $e^{\nu\lambda_1 t}$ gives 
$$
\frac{d }{d t}[e^{\nu\lambda_1 t}(\| \u_n^\alpha\|^2+\alpha^2\|A^{\frac{1}{2}}\u^\alpha_n\|^2)]  \le e^{\nu\lambda_1 t}  \frac{1}{\nu \lambda_1 } \|\f\|^2.
$$
Upon integration, we obtain
$$
\begin{array}{rcl}
\|\u_n^\alpha(t)\|^2+\alpha^2\|A^{\frac{1}{2}}\u^\alpha_n(t)\|^2 &\le&\displaystyle e^{-\nu\lambda_1 t}(\|\u_0\|^2+\alpha^2\|A^{\frac{1}{2}}\u_0\|^2)
\\
&&\displaystyle+\frac{1}{\nu \lambda_1} \|\f\|^2_{L^\infty(0,\infty; L^2(\Omega))} \int_0^t e^{-\nu\lambda_1 (t-s)}{\rm d} s
\\ 
 &\le&\displaystyle e^{-\nu\lambda_1 t}(\|\u_0\|^2+\alpha^2\|A^{\frac{1}{2}}\u_0\|^2)
 \\
 &&\displaystyle+ \frac{1}{\nu^2 \lambda_1^2} (1-e^{-\nu\lambda_1 t})\|\f\|^2_{L^\infty(0,\infty; L^2(\Omega))}.
 \end{array}
$$
Thus we have 
$$
 \|\u_n^\alpha(t)\|^2+\alpha^2\|A^{\frac{1}{2}}\u^\alpha_n(t)\|^2 \le\|\u_0\|^2+\alpha^2\|A^{\frac{1}{2}}\u_0\|^2+ \frac{1}{\nu^2 \lambda_1^2}\|\f\|^2_{L^\infty(0,\infty; L^2(\Omega))}:=E_{1,\infty}.
$$
It remains to prove (\ref{first-energy-infty2}). Let us integrate (\ref{lm14-lab1}) over $(t_0, t)$ to obtain
\begin{align*}
\|\u_n^\alpha(t)\|^2&+\alpha^2\|A^{\frac{1}{2}}\u_n^\alpha(t)\|^2+\nu\int_{t_0}^t(\|A^{\frac{1}{2}} \u_n^\alpha(s)\|^2+\alpha^2\|A\u^\alpha_n(s)\|^2){\rm d} s   
\\
&\le \|\u_n^\alpha(t_0)\|^2 + \alpha^2 \|A^{\frac{1}{2}}\u_n^\alpha(t_0)\|^2+ \frac{1}{\nu \lambda_1 }\int_{t_0}^t \|\f(s)\|^2{\rm d} s
\\
&\le E_{1,\infty}+\frac{1}{\nu \lambda_1}\|\f\|^2_{L^\infty(0,\infty;L^2(\Omega))}(t-t_0)
\\
&\le E_{1,\infty}(1+\nu\lambda_1(t-t_0)).
\end{align*}
Therefore, 
$$
\nu\int_{t_0}^t(\|A^{\frac{1}{2}}\u^\alpha_n(s)\|^2+\alpha^2\|A\u^\alpha_n(s)\|^2){\rm d} s\le E_{1,\infty}(1+\nu\lambda_1(t-t_0)).
$$
It completes the proof. 
\end{proof}

\begin{lemma} [Second energy estimates for $\u^\alpha_n$] Let $T=\infty$. There exists a positive constant $E_{2,\infty}=E_{2,\infty}(\u_0,\f, \nu, T, \Omega,\alpha)$ such that the Galerkin approximation $\u^\alpha_n$ defined by problem (\ref{Galerkin}) satisfies
\begin{equation}\label{second-energy-infty}
\sup_{0\le t<\infty}\left[ \|A^{\frac{1}{2}}\u_n^\alpha(t)\|^2 + \alpha \|A\u_n^\alpha(t)\|^2\right]\le E_{2,\infty}.
\end{equation}
Furthermore, we have, for all $0\le t_0\le t$,
\begin{equation}\label{second-energy-infty2}
\nu\int_{t_0}^t(\|A\u^\alpha_n(s)\|^2+\alpha^2\|A^{\frac{3}{2}}\u^\alpha_n(s)\|^2){\rm d} s\le E_{2,\infty}( 1+ E_{3,\infty}(t-t_0))+\frac{C}{\nu} \|\f\|^2_{L^\infty(0,\infty; L^2(\Omega))} (t-t_0),
\end{equation}
where $E_{3,\infty}=E_{3,\infty}(\u_0, \f, \nu, \Omega, \alpha)$.
\end{lemma}
\begin{proof} Firstly, we must drop the term $\|A\u_n^\alpha\|^2+\alpha^2 \|A^{\frac{3}{2}}\u^\alpha_n\|^2$ from (\ref{lm11-lab2}), with $E_{1,\infty}$ instead of $E_1$. Secondly, we apply Grönwall's inequality to it, for $ t -t^*\le s\le t$, with $t^*<t$ fixed, to find 
\begin{align*}
\displaystyle
\|A^{\frac{1}{2}}\u_n^\alpha(t)\|^2&+\alpha^2\|A\u^\alpha_n(t)\|^2 \le e^{\frac{C}{\nu^4} E_{1,\infty}^2 (1+\nu\lambda_1 t^*)} \times
\\
&\displaystyle\times\left\{\|A^{\frac{1}{2}}\u_n^\alpha(s)\|^2+\alpha^2\|A\u_n^\alpha(s)\|^2+ \frac{C}{\nu}\|\f\|^2_{L^\infty(0,\infty;L^2(\Omega)) }t^*\right\},
\end{align*}
where we have used (\ref{first-energy-infty2}). Finally, we integrate with respect to $s$, for $t-t^* \le s\le t$, to get
\begin{align*}\label{lm11-lab2}
\displaystyle
\|A^{\frac{1}{2}}\u_n^\alpha(t)\|^2&+\alpha^2\|A\u^\alpha_n(t)\|^2\le e^{\frac{C}{\nu^4} E_{1,\infty}^2 (1+\nu\lambda_1 t^*)}\times
\\
\displaystyle&\times\left\{\frac{1}{t^*}\int_{t-t^*}^{t}(\|A^{\frac{1}{2}}\u_n^\alpha(s)\|^2+\alpha^2\|A\u_n^\alpha(s)\|^2){\rm d} s+ \frac{C}{\nu}\|\f\|^2_{L^\infty(0,\infty;L^2(\Omega))} t^*\right\}
\\
&\le e^{\frac{C}{\nu^4} E_{1,\infty}^2 (1+\nu\lambda_1 t^*)} \left\{\frac{1}{t^*} E_{1,\infty}(1+\lambda\nu t^*)+\frac{C}{\nu}\|\f\|^2_{L^\infty(0,\infty;L^2(\Omega))} t^*\right\}:=E_{2,\infty},
\end{align*}
where we have again used (\ref{first-energy-infty2}). Therefore, we have that (\ref{second-energy-infty}) holds for $t>t^*$. To fill the gap for $[0,t^*]$, we take into account (\ref{second-energy}) and select $t^*$ small enough such that $E_{2}\le E_{2,\infty}$, which is, of course, always possible.

In order to obtain estimate (\ref{second-energy-infty2}), we integrate (\ref{lm11-lab2}) over $(t_0, t)$ and use (\ref{first-energy-infty}) and (\ref{second-energy-infty}). Thus, we get
$$
\begin{array}{rcl}
\displaystyle
\nu\int_{t_0}^t(\|A\u^\alpha_n(s)\|^2+\alpha^2\|A^{\frac{3}{2}}\u^\alpha_n(s)\|^2){\rm d}s 
&\le&\displaystyle \frac{C}{\nu^3} E_{1,\infty} E_{2, \infty}^2 (t-t_0)+\frac{C}{\nu} \|\f\|^2_{L^\infty(0,\infty; L^2(\Omega))} (t-t_0) + E_{2,\infty} .
\\
&\le&\displaystyle E_{2,\infty}( 1+ E_{3,\infty}(t-t_0))+\frac{C}{\nu} \|\f\|^2_{L^\infty(0,\infty; L^2(\Omega))} (t-t_0),
\end{array}
$$
where we have denoted 
$$
E_{3,\infty}:=\frac{C}{\nu^3} E_{1,\infty}E_{2,\infty}.
$$ 
\end{proof}
Using Lemma $4.1$ in \cite{Braz-Rojas}, the following corollary is derived.
\begin{corollary} Let $T=\infty$. There exists a constant $ E_{4,\infty}= E_{4,\infty}(\u_0,\f, \nu, T, \Omega,\alpha)$ such that the Galerkin approximation $\u^\alpha_n$ defined by problem  (\ref{Galerkin}) satisfies
$$
e^{-t}\int_{t_0}^t e^{s}\|A\u^\alpha_n(s)\|^2{\rm d} s\le E_{4,\infty},
$$
for all $0\le t_0 \le t$.
\end{corollary}

Analogous to the case $0<T<\infty$, one can show that there exist a subsequence $\{\u^{\alpha_j}_{n_j}\}$ and a function $\u$ such that   
$$
\begin{array}{rcl}
\u^{\alpha_j}_{n_j}\to\u &\hbox{ weakly-$\star$ in }&  L^\infty_{\rm loc}(0,\infty; D(A^{\frac{1}{2}})),  
\\
\u^{\alpha_j}_{n_j} \to \u  &\hbox{ weakly in }&L^2_{\rm loc}(0,\infty; D(A)),
\end{array}
$$
and, by a compactness result of the Aubin-Lions type, such that
$$
\u^{\alpha_j}_{n_j} \to \u\quad \hbox{ strongly in }\quad  L^2_{\rm loc}(0,\infty; D(A^{\frac{1}{2}})),
$$
with $(\alpha_j, n_j)\to (0,\infty)$ as $j\to\infty$, where $\u$ is a strong solution of the Navier-Stokes equations.  

\begin{lemma}[Second energy estimates for $\u$] Let $T=\infty$. There exists a constant $\widetilde E_{2,\infty}= \widetilde E_{2,\infty}(\u_0,\f, \nu, T, \Omega)$, which is $E_{2,\infty}$ with $\alpha=0$, such that the unique solution $\u$ to problem (\ref{NS}) satisfies
\begin{equation}\label{second-energy-infty-u}
\sup_{0\le t< \infty}\|A^{\frac{1}{2}}\u(t)\|^2\le \widetilde E_{2,\infty}.
\end{equation}
Furthermore, we have, for all $0\le t_0\le t$,
\begin{equation}\label{second-energy-infty2-u}
\nu\int_{t_0}^t\|A\u(s)\|^2{\rm d} s\le \widetilde E_{2,\infty}( 1+ \widetilde E_{3,\infty}(t-t_0))+\frac{C}{\nu} \|\f\|^2_{L^\infty(0,\infty; L^2(\Omega))} (t-t_0),
\end{equation}
where $\widetilde E_{3,\infty}=\widetilde E_{3,\infty}(\u_0, \nu, \f, \Omega)$, which is $E_{3,\infty}$  with $\alpha=0$.
\end{lemma}
Using Lemma $4.1$ in \cite{Braz-Rojas}, the following corollary is derived.
\begin{corollary} Let $T=\infty$. There exists a constant $\widetilde E_{4,\infty}=\widetilde E_{4,\infty}(\u_0,\f, \nu, T, \Omega)$ such that the unique solution $\u$ to problem (\ref{NS}) satisfies
$$
e^{-t}\int_{t_0}^t e^{s}\|A\u(s)\|^2{\rm d} s\le \widetilde E_{4,\infty},
$$
for all $0\le t_0 \le t$.
\end{corollary}
\subsection{Perturbations }

Let us introduce here the concept of the $\L^2(\Omega)$ stability of the solution $\u$ to the Navier-Stokes equations (\ref{NS}) analogous to that of \cite{Salvi}.        

\begin{definition}\label{def:perturbation}
A function $\boldsymbol\zeta$, defined for all $t\ge t_0$, is called a perturbation of $\u$ if $\u+\boldsymbol\zeta$ is a solution of
(\ref{NS}) with $\boldsymbol{\zeta}=\boldsymbol 0$ on $\partial \Omega$. That is, for a fixed $t_0\ge
0$,  $\boldsymbol{\zeta}$ is a solution of the problem
\begin{equation}\label{eq:perturbation}
\left\{
\begin{array}{rll}
\displaystyle
\frac{d}{dt}\boldsymbol{\zeta}+ \nu A\boldsymbol{\zeta} +B(\u,\boldsymbol{\zeta})+B(\boldsymbol{\zeta},\u)
+B(\boldsymbol{\zeta},\boldsymbol{\zeta})&=&\boldsymbol{0},
\\
\boldsymbol{\zeta}(t_0)&=&\boldsymbol{\zeta}_0,
\end{array}
\right.
\end{equation}
for all $t\ge t_0$.
\end{definition}
\begin{definition}\label{def:stability} A solution $\u$ to the Navier-Stokes equations (\ref{NS}) is said to be exponentially stable in the $\L^2(\Omega)$ norm if there exist two positive numbers $M$ and $B$ such that for every $t_0>0$ and every $\boldsymbol{\zeta}_0\in D(A^{\frac{1}{2}})$, the perturbation problem (\ref{eq:perturbation}) is uniquely solvable and its solution satisfies 
$$\|\boldsymbol{\zeta}(t)\|^2\le B\|\boldsymbol{\zeta}_0\|^2 e^{-M(t-t_0)},$$
for all $t\ge t_0$.
\end{definition}

The global-in-time existence and uniqueness of perturbations to (\ref{eq:perturbation}) can be established by energy methods from the theory of the Navier-Stokes equations. In particular, this is possible due to the fact that the strong solution $\u$ to the two-dimensional Navier-Stokes equations (\ref{NS}) exist globally. 

The following lemma shows that the strong solution $\u$ to the Navier-Stokes equations are exponentially stable in the sense of the $\L^2(\Omega)$ norm. The proof can be found in \cite[Theorem 2.1]{Heywood-Rannacher_1986}. 
\begin{lemma}\label{le:perturbation} There exist positive numbers $B$ and $M$ such that for every $\boldsymbol{\zeta}_0\in D(A^{\frac{1}{2}})$  and every $t_0\ge 0$, there exists a unique perturbation $\boldsymbol\zeta$ to problem (\ref{eq:perturbation}) satisfying
\begin{equation}\label{first-energy-z}
\|\boldsymbol{\zeta}(t)\|^2 \le B \| \boldsymbol{\zeta}_0\|^2e^{-M(t-t_0)}.
\end{equation}
for all $ t\ge t_0$.
Furthermore, we have
\begin{equation}\label{second-energy-z}
\|A^{\frac{1}{2}} \boldsymbol{\zeta}(t)\|^2\le B \| A^{\frac{1}{2}} \boldsymbol{\zeta}_0\|^2e^{-M(t-t_0)},
\end{equation}
for all $ t\ge t_0$.
\end{lemma}
Let us denote $P_{1,\infty}=B \| \boldsymbol{\zeta}_0\|^2 $ and $P_{2,\infty}=B \| A^{\frac{1}{2}} \boldsymbol{\zeta}_0\|^2 $ for later use.

\begin{remark} In \cite{Heywood-Rannacher_1986} and \cite{Heywood-Rannacher_1986_II} it was showed that the $\L^2(\Omega)$ and $\H^1(\Omega)$ stability are equivalent. The former is required to derive global-in-time error estimates in the $L^\infty(0,T; \H)$ norm while the latter in the $L^\infty(0,\infty; D(A^{\frac{1}{2}}))$ norm. 
\end{remark}

\begin{corollary} It also follows that 
\begin{equation}\label{stabH2-perturbation}
\nu\int_{t_0}^t \|A\boldsymbol\zeta (s)\|^2{\rm d} s\le P_{3,\infty}(1+ P_{4,\infty} (t-t_0)).
\end{equation}
for all $0\le t_0\le t$, where 
$$
P_{3,\infty}= max\{\widetilde E_{2,\infty}, P_{2,\infty}\}
$$
and
$$
P_{4,\infty}=\frac{C}{\nu^3}(\widetilde E_{1,\infty} \widetilde E_{2,\infty} +  P_{1,\infty}  P_{2,\infty}+ P_{1,\infty}\widetilde E_{2,\infty}).
$$
\end{corollary}
\begin{proof} Estimate (\ref{stabH2-perturbation}) is easily obtained from 
$$
\begin{array}{rcl}
\displaystyle
\frac{d}{dt}\|A^{\frac{1}{2}}\boldsymbol\zeta\|^2+\nu\|A\boldsymbol\zeta\|^2
&\le&\displaystyle \frac{C}{\nu^4} (\|\u\|^2 \|A^{\frac{1}{2}}\u\|^2+\|\boldsymbol\zeta\|^2\|A^{\frac{1}{2}}\boldsymbol\zeta\|^2)\|A^{\frac{1}{2}}\boldsymbol\zeta\|^2
\\
&&\displaystyle+\frac{C}{\nu^4}\|\boldsymbol\zeta\|^2 \|A^{\frac{1}{2}}\u\|^4,
\end{array}
$$ 
by integrating over $(t, t_0)$, which is deduced by using (\ref{BL4L4L2}) and (\ref{BLinfL2L2}).
\end{proof}
\subsection{Further results}
Recall that $\u -\u_n^\alpha = \e_n -\z^\alpha_n$ where $\e_n =\u-P_n\u_n=P^\perp_n\u $ and $\z_n^\alpha =\u_n^\alpha -P_n\u$. In the course of our analysis we shall require further estimates for $\z^\alpha_n$. 
\begin{lemma}\label{lm:stabH1-z} Suppose that there exists $K_{2,\infty}=K_{2,\infty}(\u_0,\f, \nu, \Omega)>0$ such that 
$$\|\z^\alpha_n(t)\|^2\le K_{2,\infty} (\lambda^{-\frac{1}{2}}_1\alpha^2+\lambda_{n+1}^{-\frac{3}{2}})$$ holds  for all  $t\in[0, t^*]$. Then there exist $R_{\infty}=R_{\infty}(\u_0,\f, \nu,\Omega)>0$, $n_0\in\mathds{N}$ and $\alpha_0>0$  such that 
\begin{equation}\label{stabH1-z}
\|A^{\frac{1}{2}}\z^\alpha_n(t)\|^2< R_{\infty}
\end{equation}
holds for all $t\in[0, t^*]$, provided $n>n_0$ and  $\alpha<\alpha_0$. 
\end{lemma}

\begin{proof} We have by (\ref{th1-lab4}) that 
$$
\begin{array}{rcl}
\displaystyle
\frac{d}{dt} \|\z_n^\alpha\|^2+\nu\|A^{\frac{1}{2}} \z_n^\alpha\|^2&\le&
\displaystyle\frac{C}{\nu} (\lambda_1^{-\frac{1}{2}}\alpha^2+\lambda_{n+1}^{-\frac{3}{2}}) \left[  \lambda_1^{-\frac{1}{2}} ( K_{2,\infty} \lambda_1^{-\frac{1}{2}} + \widetilde E_{2,\infty})  \|A\u\|^2\right.
\\
&&\left.+E_{2,\infty}  (\lambda_1^{-\frac{1}{2}}+\alpha)\|A \u^\alpha_n\|^2+\lambda_1^{\frac{1}{2}}\|\f\|^2\right]
\\
&\le&\displaystyle\frac{C}{\nu}(\lambda_1^{-\frac{1}{2}}\alpha^2+\lambda_{n+1}^{-\frac{3}{2}})[W_1( \|A\u\|^2+\|A \u^\alpha_n\|^2)+\lambda_1^{\frac{1}{2}}\|\f\|^2],
\end{array}
$$
where
$$
W_1:=\max\{\lambda_1^{-\frac{1}{2}} ( K_{2,\infty} \lambda_1^{-\frac{1}{2}} + \widetilde E_{2,\infty}),  E_{2,\infty}  (\lambda_1^{-\frac{1}{2}}+\alpha)\}.
$$
Then if we multiply by $e^t$, we arrive at
$$
\begin{array}{rcl}
\displaystyle
\frac{d}{dt}(e^{t} \|\z_n^\alpha\|^2)-e^t\|\z^\alpha_n\|^2+\nu e^t\|A^{\frac{1}{2}} \z_n^\alpha\|^2&\le&
\displaystyle\frac{C}{\nu} e^{t}  (\lambda_1^{-\frac{1}{2}}\alpha^2+\lambda_{n+1}^{-\frac{3}{2}})[W_1( \|A\u\|^2+\|A \u^\alpha_n\|^2)+\lambda_1^{\frac{1}{2}}\|\f\|^2].
\end{array}
$$
Integrating over $(0,t)$, with $t\le t^*$, and multiplying by $e^{-t}$, we obtain
$$
\begin{array}{rcl}
\displaystyle
\nu e^{-t}\int_{0}^t e^s\|A^{\frac{1}{2}}\z_n^\alpha(s)\|^2{\rm d} s &\le&
\displaystyle
e^{-t}\|\z_n^\alpha(0)\|^2+e^{-t} \int_{0}^t e^s \|\z_n^\alpha(s)\|^2 {\rm d} s
\\
&&\displaystyle
+ \frac{C}{\nu} W_1(\lambda_1^{-\frac{1}{2}}\alpha^2+\lambda_{n+1}^{-\frac{3}{2}})e^{-t}\int_{0}^t e^s( \|A\u(s)\|^2+\|A \u^\alpha_n(s)\|^2){\rm d} s
\\
&&\displaystyle+\frac{C}{\nu} \lambda_1^{\frac{1}{2}}(\lambda_1^{-\frac{1}{2}}\alpha^2+\lambda_{n+1}^{-\frac{3}{2}}) e^{-t}\int_{0}^t e^s\|\f(s)\|^2{\rm d} s
\\
&\le& 
\displaystyle
[K_{2,\infty}+ \frac{C}{\nu}(W_1(E_{4,\infty}+\widetilde E_{4,\infty})+\lambda_1^{\frac{1}{2}}\|\f\|^2_{L^\infty(0,\infty;L^2(\Omega))})] (\lambda_1^{-\frac{1}{2}}\alpha^2+\lambda_{n+1}^{-\frac{3}{2}}),
\end{array}
$$
where we have used the fact that $\z_n^\alpha(0)=\boldsymbol{0}$ and our hypothesis. More compactly, we write 
\begin{equation}\label{lm23-lab1}
 e^{-t}\int_{0}^t e^s\|A^{\frac{1}{2}}\z_n^\alpha(s)\|^2{\rm d} s\le W_2  (\lambda_1^{-\frac{1}{2}}\alpha^2+\lambda_{n+1}^{-\frac{3}{2}}).
\end{equation}

Next we take the $\L^2(\Omega)$-inner product of (\ref{eq:z_II}) with $A\z^\alpha_n$ to get
\begin{equation}\label{lm23-lab2}
\begin{array}{rcl}
\displaystyle
\frac{1}{2}\frac{d}{dt} \|A^{\frac{1}{2}}\z_n^\alpha\|^2&+&\nu\|A\z_n^\alpha\|^2
= -(B(\u, \z^\alpha_n), A\z^\alpha_n) 
\\
&-&(B(\z_n^\alpha, \u), A\z^\alpha_n)-(B(\z^\alpha_n, \z^\alpha_h), A\z^\alpha_n)+(B(\z^\alpha_n, \e_n), A\z^\alpha_n)
\\
&-&( B(\e_n, \z^\alpha_n), A\z_n^\alpha)+(B(\u, \e_n), A\z^\alpha_n)-( B(\e_n, \boldsymbol\eta_n), A\z_n^\alpha)
\\
&+&((I+\alpha^2 A )^{-1}P_n (B(\u^\alpha_n, \u^\alpha_n)-B(\u^\alpha_n,\v^\alpha_n)), A\z_n^\alpha)
\\
&+&((I+\alpha^2 A)^{-1}-I) P_n B(\u^\alpha_n, \u^\alpha_n), A\z^\alpha_n)
\\
&+&((I+\alpha^2 A)^{-1} P_n B^\star(\u_n^\alpha, (I+\alpha^2 A)\u^\alpha_n), A\z^\alpha_n)
\\
&+&(((I+\alpha^2 A)^{-1}-I)P_n\f, A\z^\alpha_n)
\\
&:=&\displaystyle\sum_{i=1}^{11}L_i.
\end{array}
\end{equation}
We shall bound each of the terms on the right-hand side of (\ref{lm23-lab2}) separately. Let $\varepsilon$ be a positive constant (to be adjusted below).  Thus, from (\ref{BL4L4L2}), we have:
$$
\begin{array}{rcl}
L_1&\le& C \|\u\|^{\frac{1}{2}} \|A^{\frac{1}{2}}\u\|^{\frac{1}{2}}  \|A^{\frac{1}{2}}\z_n^\alpha\|^{\frac{1}{2}}\|A\z^\alpha_n\|^{\frac{3}{2}}
\\
&\le&\displaystyle \frac{C_\varepsilon}{\nu^3} \widetilde E_{1,\infty} \widetilde E_{2,\infty} \|A^{\frac{1}{2}}\z^\alpha_n\|^2+\nu \varepsilon \|A\z_n^\alpha\|^2, 
\end{array}
$$ 
$$
\begin{array}{rcl}
L_3
&\le&\displaystyle \frac{C_\varepsilon}{\nu^3} \|\z^\alpha_n\|^2  \|A^{\frac{1}{2}}\z^\alpha_n\|^4+\nu \varepsilon \|A\z_n^\alpha\|^2, 
\end{array}
$$ 

$$
\begin{array}{rcl}
L_5
&\le&\displaystyle \frac{C_\varepsilon}{\nu^3} \widetilde E_{1,\infty} \widetilde E_{2,\infty} \|A^{\frac{1}{2}}\z^\alpha_n\|^2+\nu \varepsilon \|A\z_n^\alpha\|^2,
\end{array}
$$
where we have used (\ref{stabL2})  and (\ref{stabH1}) in bounding $L_5$. In view of (\ref{BLinfL2L2}) and (\ref{Poincare}), we obtain the bounds for $L_2$ and $L_4$:   
$$
\begin{array}{rcl}
L_2&\le& C \|\z^\alpha_n\|^{\frac{1}{2}}  \|A^{\frac{1}{2}}\u\| \|A\z_n^\alpha\|^{\frac{3}{2}} 
\\
&\le& C \lambda^{-\frac{1}{4}}_1 \|A^{\frac{1}{2}}\z^\alpha_n\|^{\frac{1}{2}}  \|A^{\frac{1}{2}}\u\| \|A\z_n^\alpha\|^{\frac{3}{2}} 
\\
&\le&\displaystyle \frac{C_\varepsilon}{\nu^3} \lambda_{1}^{-1}\widetilde E_{2,\infty}^2 \|A^{\frac{1}{2}}\z^\alpha_n\|^2+\nu \varepsilon \|A\z_n^\alpha\|^2, 
\end{array}
$$ 
$$
\begin{array}{rcl}
L_4
&\le&\displaystyle \frac{C_\varepsilon}{\nu^3} \lambda_1^{-1} \widetilde E_{2,\infty}^2 \|A^{\frac{1}{2}}\z^\alpha_n\|^2+\nu \varepsilon \|A\z_n^\alpha\|^2.
\end{array}
$$

It follows, again using (\ref{BL4L4L2}) and also (\ref{errorH1H2yL2H2}) and (\ref{stabH2}), that
$$
\begin{array}{rcl}
L_6&\le& C \|\u\|^{\frac{1}{2}} \|A^{\frac{1}{2}}\u\|^{\frac{1}{2}} \lambda^{-\frac{1}{4}}_{n+1} \|A\e_n\|  \|A\z_n^\alpha\|  
\\
&\le&\displaystyle\frac{C_\varepsilon}{\nu} \widetilde E_{1,\infty}^\frac{1}{2} \widetilde E_{2,\infty}^{\frac{1}{2}} \lambda_{n+1}^{-\frac{1}{2}}  \|A\u\|^2+\varepsilon \nu\|A\z_n^\alpha\|^2.
\end{array}
$$
The bound for $L_7$ proceeds by taking into account (\ref{BLinfL2L2}), (\ref{errorH1H2yL2H2}), (\ref{stabH1}) and  (\ref{stabH2}):
$$
\begin{array}{rcl}
L_7
&\le&\displaystyle\frac{C_\varepsilon}{\nu} \widetilde E_{2,\infty} \lambda_{n+1}^{- 1}  \|A\u\|^2+\varepsilon \nu\|A\z_n^\alpha\|^2.
\end{array}
$$
We estimate $L_8$ analogously as $J_8$. Thus we have by (\ref{I-Ja=aAJa}), (\ref{BLinfL2L2}) and (\ref{Resolvent-alphaH1}) that   
$$
\begin{array}{rcl}
L_8&=&\alpha^2(B(\u^\alpha_n, (I+\alpha^2 A)^{-1} A\z_n^\alpha),  A\u_n^\alpha )
\\
&\le &\alpha \|\u^\alpha_n\|^{\frac{1}{2}} \|A\u^\alpha_n\|^{\frac{3}{2}}  \|A\z_n^\alpha\|
\\
&\le&\displaystyle \frac{C_\varepsilon}{\nu}E_{1,\infty}^\frac{1}{2} E_{2,\infty}^\frac{1}{2}\alpha  \|A\u_n^\alpha\|^2+\nu\varepsilon \|A\z_n^\alpha\|^2.
\end{array}
$$
The term $L_9$ is also treated as its counterpart $J_9$. Then, by Lemma \ref{le:Temam} , we get 
$$
\begin{array}{rcl}
L_9&=&\alpha ( A^{\frac{1}{2}} B(\u^\alpha_n, \u^\alpha_n), (\alpha A)^{\frac{1}{2}}(I+\alpha^2 A)^{-1}A\z^\alpha_n)
\\
&\le & \alpha \|A^{\frac{1}{2}} B(\u^\alpha_n, \u^\alpha_n)\| \|A\z^\alpha_n\|
\\
&\le & \alpha (\|A^{\frac{1}{2}}\u^\alpha_n\|_{L^4(\Omega)}^2 + \|\u^\alpha_n\|_{L^\infty(\Omega)}^\frac{1}{2} \|A\u^\alpha_n\|^{\frac{3}{2}}) \|A\z^\alpha_n\|.
\end{array}
$$
Next Gagliardo-Nirenberg's and Agmon's inequalities give 

$$ L_9 \le\frac{C_\varepsilon}{\nu}\alpha (\alpha E_{2,\infty}+E_{1,\infty}^{\frac{1}{2}} E_{2,\infty}^{\frac{1}{2}}  ) \|A\u^\alpha_n\|^2+\varepsilon\nu\|A\z^\alpha_n\|^2.$$
We proceed in the manner of $J_{10}$ to obtain a bound for $L_{10}$, but using (\ref{BL2L2Linf}):
$$
\begin{array}{rcl}
L_{10}
&\le& \displaystyle\frac{C_\varepsilon}{\nu}E_{1,\infty}^{\frac{1}{2}} E_{2,\infty}^{\frac{1}{2}} \alpha\|A\u^\alpha_n\|^2+\varepsilon \nu\|A \z_n^\alpha\|^2.
\end{array}
$$
By virtue of (\ref{ResolventH2}), we see that 
$$
\begin{array}{rcl}
L_{11}
&\le&\displaystyle\frac{C_\varepsilon}{\nu}\|\f\|^2+\varepsilon\nu \|A^{\frac{1}{2}}\z^\alpha_n\|^2.
\end{array}
$$

Assembling the estimates of the $L_i$'s into (\ref{lm23-lab2}) and adjusting  $\varepsilon$ properly, we find 
\begin{equation}\label{lm23-lab3}
\begin{array}{rcl}
\displaystyle
\frac{d}{dt} \|A^{\frac{1}{2}}\z_n^\alpha\|^2+\nu\|A\z_n^\alpha\|^2&\le&\displaystyle \frac{C}{\nu} W_3\|A^{\frac{1}{2}}\z^\alpha_n\|^2+\frac{C}{\nu} K_{2,\infty}(\lambda^{-\frac{1}{2}}_1\alpha^2+\lambda_{n+1}^{-\frac{3}{2}}) \|A^{\frac{1}{2}}\z^\alpha_n\|^4
\\
&&\displaystyle+\frac{C}{\nu}W_4(\alpha+\lambda_{n+1}^{-\frac{1}{2}}) (\|A\u\|^2+\|A\u^\alpha_n\|^2)
\\
&& \displaystyle+\frac{C}{\nu}\|\f\|^2_{L^\infty(0,\infty; L^2(\Omega))}.
\end{array}
\end{equation}
where 
$$W_3= \frac{\widetilde E_{2,\infty}}{\nu^2} (\widetilde E_{1,\infty} +\widetilde E_{2,\infty}\lambda_1^{-1}),$$
$$W_4=\max\{\widetilde E_{2,\infty}^{\frac{1}{2}}(\widetilde E_{1,\infty}^{\frac{1}{2}}  +\widetilde E_{2,\infty}^{\frac{1}{2}}\lambda_{n+1}^{-\frac{1}{2}}) , E^{\frac{1}{2}}_{2,\infty}[E^{\frac{1}{2}}_{1,\infty}+E^{\frac{1}{2}}_{2,\infty}\alpha]\}.$$

Now we claim that 
\begin{equation}\label{lm23-lab4}
\|A^{\frac{1}{2}}\z^\alpha_n( t)\|^2< R_{\infty}:=\frac{4C}{\nu} \|\f\|^2_{L^\infty(0,\infty; L^2(\Omega))}
\end{equation}
holds for all $t\in[0,t^*]$, whenever $n\ge n_0$ and $\alpha \le \alpha_0$, where $n_0$ and $\alpha_0$ will determine later. Conversely, suppose that (\ref{lm23-lab4}) fails; i.e. suppose that there must be some $n\ge n_0$ and $\alpha\le \alpha_0 $ for which there is a first time $ t'$ so that the bound is attained. That is, let $t'$ be the first time  such that  
\begin{equation}\label{lm23-lab5}
\|A^{\frac{1}{2}}\z^\alpha_n( t')\|^2=R_{\infty};
\end{equation}
hence
\begin{equation}\label{lm23-lab6}
\|A^{\frac{1}{2}}\z^\alpha_n( t)\|^2\le R_{\infty}
\end{equation}
for all $t\in[0,t']$. Next,  multiplying (\ref{lm23-lab3}) by $e^t$, integrating over $(0, t')$, and multiplying by $e^{- t'}$ successively gives 
$$
\begin{array}{rcl}
\|A^{\frac{1}{2}}\z_n^\alpha(t')\|^2&\le&\displaystyle \frac{C}{\nu} W_3 e^{-t'}\int_0^{t'} e^{s}\|A^{\frac{1}{2}}\z^\alpha_n(s)\|^2 d{\rm s}
\\
&&\displaystyle+ \frac{C}{\nu}K_{2,\infty}(\lambda^{-\frac{1}{2}}_1\alpha^2+\lambda_{n+1}^{-\frac{3}{2}})e^{-t'}\int_0^{t'} e^{s} \|A^{\frac{1}{2}}\z^\alpha_n(s)\|^4\,d{\rm s}
\\
&&\displaystyle+\frac{C}{\nu} W_4(\alpha+ \lambda_{n+1}^{-\frac{1}{2}}) e^{-t'}\int_0^{t'} e^{s}(\|A\u(s)\|^2+\|A\u^\alpha_n(s)\|^2)\, d{\rm s}
\\
&&\displaystyle+\frac{C}{\nu}\|\f\|^2_{L^\infty(0,\infty; L^2(\Omega))}.
\end{array}
$$ 
Now, form (\ref{lm23-lab1}) and (\ref{lm23-lab6}), we see that  
$$
\begin{array}{rcl}
\|A^{\frac{1}{2}}\z_n^\alpha(t')\|^2&\le&\displaystyle\frac{C}{\nu} W_3 W_2 (\lambda_1^{-\frac{1}{2}}\alpha^2+\lambda_{n+1}^{-\frac{3}{2}})+\frac{C}{\nu}K_{2,\infty}R_\infty W_2 (\lambda^{-\frac{1}{2}}_1\alpha^2+\lambda_{n+1}^{-\frac{3}{2}})^2
\\
&&\displaystyle+\frac{C}{\nu}W_4(\alpha+\lambda_{n+1}^{-\frac{1}{2}})(\widetilde E_{4,\infty}+ E_{4,\infty}) +\frac{C}{\nu}\|\f\|^2_{L^2(0,\infty; L^2(\Omega))}.
\end{array}
$$
Therefore, if we select $n_0\in \mathds{N}$ and $\alpha_0>0$ sufficiently large such that 
$$
W_3 W_2 (\lambda_1^{-\frac{1}{2}}\alpha^2+\lambda_{n+1}^{-\frac{3}{2}})<\|\f\|^2_{L^2(0,\infty; L^2(\Omega))}, 
$$
$$
K_{2,\infty}R_\infty W_2 (\lambda^{-\frac{1}{2}}_1\alpha^2+\lambda_{n+1}^{-\frac{3}{2}})^2<\|\f\|^2_{L^2(0,\infty; L^2(\Omega))} 
$$
and
$$
W_4(\alpha+\lambda_{n+1}^{-\frac{1}{2}})(E_{4,\infty}+\widetilde E_{4,\infty})<\|\f\|^2_{L^2(0,\infty; L^2(\Omega))} 
$$
we arrive at 
$$
\|A^{\frac{1}{2}}\z_n^\alpha(t')\|^2< R_{\infty},
$$
which is a contradiction with (\ref{lm23-lab5}). Thus, (\ref{lm23-lab4}) cannot fail.
\end{proof}
Next, we write (\ref{eq:perturbation}) as 
\begin{equation}\label{eq:perturbation_II}
\frac{d}{dt} P_n\boldsymbol\zeta +\nu  A P_n \boldsymbol\zeta=-P_nB(\u, \boldsymbol\zeta)-P_nB(\boldsymbol\zeta, \u)-P_nB(\boldsymbol\zeta,\boldsymbol\zeta).
\end{equation}
Using the fact that $\boldsymbol\zeta=P_n\boldsymbol\zeta+P_n^\perp\boldsymbol\zeta$, we split the right hand side of (\ref{eq:perturbation_II}) as follows:
\begin{equation}\label{eq:perturbation_III}
\begin{array}{rcl}
\displaystyle
\frac{d}{dt} P_n\boldsymbol\zeta +\nu A P_n\boldsymbol\zeta&=&-P_nB(\u, P_n\boldsymbol\zeta)-P_nB(\u, P_n^\perp\boldsymbol\zeta)
\\
&&
-P_nB(P_n\boldsymbol\zeta, \u)-P_nB(P_n^\perp\boldsymbol\zeta, \u)
\\
&&
-P_nB(P_n\boldsymbol\zeta,P_n\boldsymbol\zeta)-P_nB(P_n\boldsymbol\zeta,P_n^\perp\boldsymbol\zeta)
\\
&&-P_nB(P_n^\perp\boldsymbol\zeta,P_n\boldsymbol\zeta)-P_nB(P_n^\perp\boldsymbol\zeta,P_n^\perp\boldsymbol\zeta).
\end{array}
\end{equation}
Let $\w^\alpha_n=\z^\alpha_n-P_n\boldsymbol\zeta$. Then, subtracting (\ref{eq:perturbation_III}) from (\ref{eq:z_II}) gives
\begin{equation}\label{eq:w}
\begin{array}{rcl}
\displaystyle
\frac{d}{dt}\w^\alpha_n+\nu A \w^\alpha_n&=&-P_nB(\u, \w^\alpha_n)-P_nB(\w^\alpha_n, \u)
\\
&&-P_nB(\z^\alpha_n, \w^\alpha_n)+P_nB(\w^\alpha_n, P_n\boldsymbol\zeta)
\\
&&+P_nB(\u, P_n^\perp\boldsymbol\zeta)+P_nB(P_n^\perp\boldsymbol\zeta, \u)
\\
&&+P_nB(P_n\boldsymbol\zeta,P_n^\perp\boldsymbol\zeta)+P_nB(P_n^\perp\boldsymbol\zeta,P_n\boldsymbol\zeta)
\\
&&+P_nB(P_n^\perp\boldsymbol\zeta,P_n^\perp\boldsymbol\zeta)+P_nB(\z^\alpha_n, \e_n)
\\
&&+P_nB(\e_n,\z^\alpha_n)+P_nB(\u, \e_n)+P_nB(\e_n, \boldsymbol\eta_n)
\\
&&+(I+\alpha^2 A )^{-1}P_n (B(\u^\alpha_n, \u^\alpha_n)-B(\u^\alpha_n,\v^\alpha_n))
\\
&&+((I+\alpha^2 A)^{-1}-I) P_nB(\u^\alpha_n, \u^\alpha_n)
\\
&&+(I+\alpha^2 A)^{-1} P_nB^\star(\u_n^\alpha, \v^\alpha_n)
\\
&&+((I+\alpha^2 A)^{-1}-I)P_n\f.
\end{array}
\end{equation}
\begin{lemma}\label{lm:stabL2-w } Under the conditions of Lemma \ref{lm:stabH1-z},  it follows that, for $t_0\ge 0$,
\begin{equation}\label{stabL2-w}
\begin{array}{rcl}
\|\w^\alpha_n(t)\|^2&\le& e^{\frac{C_0}{\nu} (\widetilde E_{2, \infty}^2+P_{2, \infty}^2) (t-t_0)}\times
\\
&&\displaystyle\times\left\{ \|\w^\alpha_n(t_0)\|^2+\frac{C_1}{\nu}(\lambda^{-\frac{1}{2}}_1\alpha^{2}+\lambda_{n+1}^{-\frac{3}{2}})\int_{t_0}^t g(s)\, {\rm d}s\right\},
\end{array}
\end{equation}
for all $t\ge t_0$, where $g(s)= (\widetilde E_{2,\infty}+R_\infty)\lambda_{1}^{-\frac{1}{2}}\|A\u\|^2+E_{2, \infty}(\alpha+\lambda_1^{-\frac{1}{2}})\|A\u^\alpha_n\|^2+ \lambda_1^{-\frac{1}{2}}P_{2,\infty}\|A\boldsymbol\zeta\|^2+\|\f\|^2.$
\end{lemma}
\begin{proof} Let us take the $\L^2(\Omega)$-inner product of (\ref{eq:w}) with $\w^\alpha_n$ to obtain 
\begin{equation}\label{lm24-lab1}
\begin{array}{rcl}
\displaystyle
\frac{1}{2}\frac{d}{dt}\|\w^\alpha_n\|^2+\nu \|A^{\frac{1}{2}} \w^\alpha_n\|^2&=&-(B(\u, \w^\alpha_n),\w^\alpha_n)-(B(\w^\alpha_n, \u), \w^\alpha_n)
\\
&&-(B(\z^\alpha_n, \w^\alpha_n), \w^\alpha_n)+(B(\w^\alpha_n, P_n\boldsymbol\zeta), \w^\alpha_n)
\\
&&+B(\u, P_n^\perp\boldsymbol\zeta), \w^\alpha_n)+(B(P_n^\perp\boldsymbol\zeta, \u), \w^\alpha_n)
\\
&&+(B(P_n\boldsymbol\zeta,P_n^\perp\boldsymbol\zeta), \w^\alpha_n)
+(B(P_n^\perp\boldsymbol\zeta,P_n\boldsymbol\zeta), \w^\alpha_n)
\\
&&+(B(P_n^\perp\boldsymbol\zeta,P_n^\perp\boldsymbol\zeta), \w^\alpha_n)+(B(\z^\alpha_n, \e_n), \w^\alpha_n)
\\
&&+(B(\e_n,\z^\alpha_n), \w^\alpha_n)+(B(\u, \e_n), \w^\alpha_n)+(B(\e_n, \boldsymbol\eta_n),\w^\alpha_n)
\\
&&+((I+\alpha^2 A )^{-1}P_n (B(\u^\alpha_n, \u^\alpha_n)-B(\u^\alpha_n,\v^\alpha_n)), \w^\alpha_n)
\\
&&+(((I+\alpha^2 A)^{-1}-I) P_nB(\u^\alpha_n, \u^\alpha_n), \w^\alpha_n)
\\
&&+((I+\alpha^2 A)^{-1} P_nB^\star(\u_n^\alpha, \v^\alpha_n),\w^\alpha_n)
\\
&&+(((I+\alpha^2 A)^{-1}-I)P_n\f, \w^\alpha_n)
\\
&:=&\displaystyle\sum_{i=1}^{17} M_i.
\end{array}
\end{equation}
We first observe that $M_1$ and $M_3$ vanish by (\ref{Skew_Symmetric_B-bis}). From (\ref{BL4L2L4}), we bound
$$
M_2
\le \frac{C_\varepsilon}{\nu} \widetilde E_{2,\infty} \|\w^\alpha_n\|^2+\varepsilon\nu \|A^{\frac{1}{2}} \w^\alpha_n\|^2,
$$
$$
M_4 
\le \frac{C_\varepsilon}{\nu} P_{2,\infty} \|\w^\alpha_n\|^2+\varepsilon\nu \|A^{\frac{1}{2}} \w^\alpha_n\|^2.
$$
Combining successively (\ref{Skew_Symmetric_B}), (\ref{BLinfL2L2}), (\ref{errorL2H1}),  (\ref{errorH1H2yL2H2}), (\ref{stabH1}), (\ref{stabH2}) and (\ref{Poincare}), we see easily that   
$$
\begin{array}{rcl}
M_5
&\le&\displaystyle\frac{C_\varepsilon}{\nu} \lambda_{1}^{-\frac{1}{2}}\lambda^{-\frac{3}{2}}_{n+1} (\widetilde E_{2,\infty}\|A\u\|^2+ P_{2, \infty}\|A\boldsymbol\zeta\|^2)+\varepsilon \nu\|A^{\frac{1}{2}}\w_n^\alpha\|^2,
\end{array}
$$
$$
\begin{array}{rcl}
M_7
&\le&\displaystyle\frac{C_\varepsilon}{\nu} \lambda_{1}^{-\frac{1}{2}}\lambda^{-\frac{3}{2}}_{n+1} P_{2, \infty} \|A\boldsymbol\zeta\|^2+\varepsilon \nu\|A^{\frac{1}{2}}\w_n^\alpha\|^2,
\end{array}
$$
$$
\begin{array}{rcl}
M_{12}
&\le&\displaystyle\frac{C_\varepsilon}{\nu}\lambda_{1}^{-\frac{1}{2}} \lambda^{-\frac{3}{2}}_{n+1} \widetilde E_{2, \infty} \|A\u\|^2+\varepsilon \nu\|A^{\frac{1}{2}}\w_n^\alpha\|^2.
\end{array}
$$
As before, but utilizing (\ref{BL2L2Linf}), instead of (\ref{BLinfL2L2}), there are no difficulties in finding that 
$$
M_6
\le\frac{C_\varepsilon}{\nu} \lambda_{1}^{-\frac{1}{2}} \lambda^{-\frac{3}{2}}_{n+1} (\widetilde E_{2,\infty}\|A\u\|^2+ P_{2, \infty}\|A\boldsymbol\zeta\|^2)+\varepsilon \nu\|A^{\frac{1}{2}}\w_n^\alpha\|^2,
$$
$$
M_8
\le\frac{C_\varepsilon}{\nu}\lambda_{1}^{-\frac{1}{2}} \lambda^{-\frac{3}{2}}_{n+1} P_{2, \infty} \|A\boldsymbol\zeta\|^2+\varepsilon \nu\|A^{\frac{1}{2}}\w_n^\alpha\|^2,
$$
$$
M_9
\le\frac{C_\varepsilon}{\nu}\lambda_{1}^{-\frac{1}{2}} \lambda^{-\frac{3}{2}}_{n+1} P_{2,\infty} \|A\boldsymbol\zeta\|^2+\varepsilon \nu\|A^{\frac{1}{2}}\w_n^\alpha\|^2,
$$
$$M_{13}\le \frac{C_\varepsilon}{\nu} \lambda^{-\frac{1}{2}}_{1} \lambda^{-\frac{3}{2}}_{n+1} \widetilde E_{2,\infty} \|A\u\|^2+\varepsilon \nu\|A^{\frac{1}{2}}\w_n^\alpha\|^2. $$
For $M_{10}$ and $M_{11}$, we use (\ref{BL4L2L4}) and (\ref{Poincare}) to get
$$
\begin{array}{rcl}
M_{10}
&\le&\displaystyle\frac{C_\varepsilon}{\nu}\lambda^{-\frac{1}{2}}_{1} \lambda^{-\frac{3}{2}}_{n+1}  R_{\infty}\|A\u\|^2+\varepsilon \nu\|A^{\frac{1}{2}}\w_n^\alpha\|^2,
\end{array}
$$
$$
\begin{array}{rcl}
M_{11}
&\le&\displaystyle\frac{C_\varepsilon}{\nu} \lambda^{-\frac{1}{2}}_{1} \lambda^{-\frac{3}{2}}_{n+1} R_{\infty} \|A\u\|^2+\varepsilon \nu\|A^{\frac{1}{2}}\w_n^\alpha\|^2,
\end{array}
$$
where we have employed estimate (\ref{stabH1-z}). Finally, the $M_{i}$'s, for $i=14, 15, 16 ,17$, are bounded exactly as the $J_i$'s, for $i=8, 9, 10, 11$, respectively. Thus, we obtain
$$
\begin{array}{rcl}
M_{14}
&\le&\displaystyle\frac{C_\varepsilon}{\nu}  \lambda_1^{-\frac{1}{2}} \alpha^3 E_{2,\infty}   \|A\u^\alpha_n\|^2+\varepsilon\nu \|A^{\frac{1}{2}}\w^\alpha_n\|^2.
\end{array}
$$ 
$$
\begin{array}{rcl}
M_{15}
&\le& \displaystyle\frac{C_\varepsilon}{\nu}  \lambda_1^{-1} \alpha^2 E_{2,\infty} \|A\u^\alpha_n\|^2 +\varepsilon\nu\|A^{\frac{1}{2}}\w^\alpha_n\|^2,
\end{array}
$$
$$
\begin{array}{rcl}
M_{16}
&\le& \displaystyle \frac{C_\varepsilon}{\nu} \lambda_1^{-\frac{1}{2}} \alpha^3  E_{2,\infty} \|A\u^\alpha_n\|^2+\varepsilon \nu\|A^{\frac{1}{2}} \w_n^\alpha\|^2,
\end{array}
$$
$$
\begin{array}{rcl}
M_{17}
&\le &
\displaystyle\frac{C_\varepsilon}{\nu} \alpha^2 \|\f\|^2+\varepsilon\nu \|A^{\frac{1}{2}}\w^\alpha_n\|^2.
\end{array}
$$
The previous estimates applied  to (\ref{lm24-lab1}) yield the bound, after choosing $\varepsilon$ correctly,  
$$
\begin{array}{rcl}
\displaystyle
\frac{d}{dt}\|\w^\alpha_n\|^2+\nu \|A^{\frac{1}{2}} \w^\alpha_n\|^2
&\le& 
\displaystyle
\frac{C}{\nu} ( \widetilde E_{2,\infty}+P_{2, \infty}) \|\w^\alpha_n\|^2 
\\
&&
\displaystyle
+\frac{C}{\nu} (\lambda_1^{-\frac{1}{2}}\alpha^{2}+\lambda^{-\frac{3}{2}}_{n+1}) g(t),
\end{array}
$$
where $$g(t)= (\widetilde E_{2,\infty}+R_\infty)\lambda_{1}^{-\frac{1}{2}}\|A\u\|^2+E_{2, \infty}(\alpha+\lambda_1^{-\frac{1}{2}})\|A\u^\alpha_n\|^2+ \lambda_1^{-\frac{1}{2}}P_{2,\infty}\|A\boldsymbol\zeta\|^2+\lambda_1^{\frac{1}{2}}\|\f\|^2.$$
Thus, Grönwall's inequality gives (\ref{stabL2-w}).
\end{proof}

Finally, from (\ref{second-energy-infty2}), (\ref{second-energy-infty2-u}), and (\ref{stabH2-perturbation}), we obtain  
$$
\begin{array}{rcl}
\|\w^\alpha_n(t)\|^2&\le& e^{\frac{C_0}{\nu} (\widetilde E_{2, \infty}+P_{2, \infty}) (t-t_0)}\times
\\
&&\displaystyle\times\left\{ \|\w^\alpha_n(t_0)\|^2+\frac{C_1}{\nu^2}S_{1,\infty} S_{2,\infty}(\lambda_1^{-\frac{1}{2}}\alpha^{2}+\lambda_{n+1}^{-\frac{3}{2}})(1+S_{3,\infty}(t-t_0))\right.
\\
&&\left.
\displaystyle\quad\quad+C_2(\frac{S_{1,\infty}}{\nu^2} +\lambda_1^{\frac{1}{2}})(\lambda_1^{-\frac{1}{2}}\alpha^{2}+\lambda_{n+1}^{-\frac{3}{2}})\|\f\|^2_{L^\infty(0,\infty; L^2(\Omega))}(t-t_0)
\right\},
\end{array}
$$
where 
$$S_{1,\infty}=\max\{(\widetilde E_{2, \infty}+R_\infty)\lambda_1^{-\frac{1}{2}}, E_{2, \infty}(\alpha+\lambda_1^{-\frac{1}{2}}), \lambda_1^{-\frac{1}{2}}P_{2,\infty}  \},$$
$$S_{2,\infty}=\max\{\widetilde E_{2,\infty},E_{2,\infty}, P_{3,\infty}\}$$
and
$$S_{3,\infty}=\max\{\widetilde E_{3,\infty},E_{3,\infty}, P_{4,\infty}\}.$$
More compactly, 
\begin{equation}\label{stabL2-w_II}
\begin{array}{rcl}
\|\w^\alpha_n(t)\|^2&\le& e^{G_{1,\infty} (t-t_0)}\times
\\
&&\displaystyle\times\left\{ \|\w^\alpha_n(t_0)\|^2+G_{2,\infty} (\lambda_1^{-\frac{1}{2}}\alpha^{2}+\lambda_{n+1}^{-\frac{3}{2}})(1+G_{3,\infty}(t-t_0))\right.
\\
&&\left.
\displaystyle\quad\quad+G_{4,\infty}(\lambda_1^{-\frac{1}{2}}\alpha^{2}+\lambda_{n+1}^{-\frac{3}{2}})\|\f\|^2_{L^\infty(0,\infty; L^2(\Omega))}(t-t_0)
\right\},
\end{array}
\end{equation}
where the $G_{i,\infty}$'s are defined in the obvious way. 
\subsection{Proof of Theorem \ref{Th2}}
We already know that the  solution $\u$ of the Navier-Stokes equations is stable in the $\L^2(\Omega)$ sense in view of Lemma \ref{le:perturbation}.  Then choose $T$  large enough that
\begin{equation}\label{th2-lab1}
B e^{- MT}\le\frac{1}{4} 
\end{equation}
and hence define
\begin{equation}\label{th2-lab2}
K_{2, \infty}:=4 e^{ G_{1, \infty} T}\{  G_{2,\infty} (1+G_{3,\infty}T)+G_{4,\infty} \|\f\|^2_{L^\infty(0,\infty;L^2(\Omega)) }T\}.
\end{equation}
For all  $n\ge n_0$ and $\alpha\le \alpha_0$ in Lemma \ref{lm:stabH1-z}, we assert
\begin{equation}\label{th2-lab4}
\|\z^\alpha_n (t)\|^2< K_{2, \infty} (\lambda_1^{-\frac{1}{2}}\alpha^2+\lambda_{n+1}^{-\frac{3}{2}}).
\end{equation}
for all $t\ge 0$. But if not, there would exist  some $n\ge n_0$ and $\alpha\le\alpha_0$ such that (\ref{th2-lab1}) fails for some time $t^*$.  Let $t^*$ be the first value of $t$ for which 
\begin{equation}\label{th2-lab5}
\|\z^\alpha_n(t^*)\|^2=K_{2, \infty} (\lambda_1^{-\frac{1}{2}}\alpha^2+\lambda_{n+1}^{-\frac{3}{2}}).
\end{equation}
As a result, we have that 
\begin{equation}\label{th2-lab5-bis}
\|\z^\alpha_n(t)\|^2\le K_{2, \infty} (\lambda_1^{-\frac{1}{2}}\alpha^2+\lambda_{n+1}^{-\frac{3}{2}})
\end{equation}
holds for all $t^*\in[0,t^*]$; therefore inequality (\ref{stabL2-w_II}) is true in view of Lemmas \ref{lm:stabH1-z} and \ref{lm:stabL2-w }.

Firstly assume $t^*\le T$. Then use inequality (\ref{stabL2-w_II}), with $t_0=0$ and $\boldsymbol{\zeta}=0$, to get, from (\ref{th2-lab2}),
$$
\begin{array}{rcl}
\|\z^\alpha_n(t^*)\|^2=\|\w^\alpha_n(t^*)\|^2&\le&\displaystyle  e^{G_{1, \infty} t^*} \left\{ G_{2,\infty}  (1+G_{3,\infty}t^*) \right.
\\
&&+ G_{4,\infty} \|\f\|^2_{L^\infty(0,\infty; L^2(\Omega))} t^*\} (\lambda_1^{-\frac{1}{2}}\alpha^2+\lambda_{n+1}^{-\frac{3}{2}})
\\
&<&\displaystyle \frac{K_{2, \infty}}{4} (\lambda_1^{-\frac{1}{2}}\alpha^2+\lambda_{n+1}^{-\frac{3}{2}}),
\end{array}
$$
which is a contradiction with (\ref{th2-lab5}). On the other hand, assume $t^*> T$. Then use
inequality  (\ref{stabL2-w_II}), with $t_0 =t^*-T$, and $\boldsymbol{\zeta}(t)$, satisfying $\boldsymbol{\zeta}(t_0)=\z^\alpha_n(t_0)$, to find
\begin{equation}\label{th2-lab6}
\begin{array}{rcl}
\|\z^\alpha_n(t^*)-P_n\boldsymbol{\zeta}(t^*)\|^2&\le&\displaystyle  e^{G_{1, \infty} T} \left\{ G_{2,\infty}  (1+G_{3,\infty}t^*) \right.
\\
&&+ G_{4,\infty} \|\f\|^2_{L^\infty(0,\infty; L^2(\Omega))} T\} (\lambda_1^{-\frac{1}{2}}\alpha^2+\lambda_{n+1}^{-\frac{3}{2}})
\\
&\le&\displaystyle \frac{K_{2, \infty}}{4} (\lambda_{1}^{-\frac{1}{2}}\alpha^2+\lambda_{n+1}^{-\frac{3}{2}}).
\end{array}
\end{equation}
Furthermore, we have, by (\ref{th2-lab1}) and (\ref{th2-lab5-bis}), that  
\begin{equation}\label{th2-lab7}
\|\boldsymbol{\zeta}(t^*)\|^2\le B \|\boldsymbol{\zeta}(t_0)\| e^{-M T}\le \frac{K_{2,\infty}}{4}  (\lambda_{1}^{-\frac{1}{2}}\alpha^2+\lambda_{n+1}^{-\frac{3}{2}}).
\end{equation}
Putting together (\ref{th2-lab6}) and (\ref{th2-lab7}) implies 
$$
\|\z^\alpha_n(t^*)\|^2\le\|\z^\alpha_n(t^*)-P_n\boldsymbol{\zeta}(t^*)\|^2+\|\boldsymbol{\zeta}(t^*)\|^2<\frac{K_{2,\infty}}{2}  (\lambda_{1}^{-\frac{1}{2}}\alpha^2+\lambda_{n+1}^{-\frac{3}{2}}), 
$$
which again is a contraduction with (\ref{th2-lab5}).

Finally, select $K_{\infty}=\max\{K_{1, \infty}, K_{2, \infty} \}$ and combine estimates (\ref{th2-lab4}) and (\ref{lm13-lab3}), with $K_{1,\infty}$, $\widetilde E_{2, \infty}$ and $\|\f\|_{L^\infty(0,\infty; L^2(\Omega))}$ instead of $K_1$, $ \widetilde E_2$ and $\|\f\|_{L^\infty(0,T; L^2(\Omega))}$, to conclude the proof. 

\section{ Concluding remarks}\label{Remarks}

\begin{enumerate}
\item If one bounds the term $L_{11}$ in Lemma \ref{lm:stabH1-z} as 
$$
\begin{array}{rcl}
L_{11}&=&\alpha^2(A(I+\alpha^2 A)^{-1} P_n\f, A\z^\alpha_n)=\alpha((\alpha^2 A)^{\frac{1}{2}}(I+\alpha^2 A)^{-1} P_n\f, A^\frac{3}{2} \z^\alpha_n)
\\
&\le &\alpha \|P_n\f\| \| A^{\frac{3}{2}}\z^\alpha_n\|\le  \alpha\|\f\| \lambda_{n}^{\frac{1}{2}} \| A\z^\alpha_n\|
\\
&\le&\displaystyle\frac{C_\varepsilon}{\nu}\alpha^2\lambda_{n}\|\f\|^2+\varepsilon\nu \|A^{\frac{1}{2}}\z^\alpha_n\|^2,
\end{array}
$$
one obtains local-in-time error estimates for the Dirichlet norm, i.e.,
$$
\sup_{0\le t \le T}\|A^{\frac{1}{2}}(\u_n^\alpha(t)-\u(t))\|^2\le K\,(\alpha +\lambda_{n+1}^{-\frac{1}{2}})
$$	
under the assumption $\alpha \lambda_1^{-\frac{1}{2}} \lambda_n<1$. In doing so, we have used the fact that $\|A^{\frac{1}{2}}P_n\u\|^2\le \lambda_{n} \|\u\|^2$ for all $\u\in D(A^{\frac{1}{2}})$.

Global-in-time error estimates for the Dirichlet norm follow by using the $\H^1(\Omega)$ stability of solutions to the Navier-Stokes equations obtained in (\ref{second-energy-z}). See \cite{Heywood}.

\item If one assumes $\frac{d}{dt}\f\in L^\infty (0, T; L^2(\Omega))$, with $0<T<\infty$ or $T=\infty$, then it follows that 
$$
\sup_{0\le t< T} \|A\u_n^\alpha(t)\|^2< \infty.
$$
Thus, optimal local- and global-in-time error estimates can be derived, i.e.,
$$
\sup_{0\le t <T}\|\u_n^\alpha(t)-\u(t)\|^2\le K\,(\lambda_{1}^{-1}\alpha^2+\lambda_{n+1}^{-2}).
$$
This argument is tedious and involves a plethora of computations. The reader has been spared such unnecessary technicalities herein.	
    

\end{enumerate}

\end{document}